\documentclass[11pt]{article}
\usepackage[left=3cm,top=2.5cm,right=3cm]{geometry}

\usepackage[utf8]{inputenc}
\usepackage{graphicx,psfrag,amsfonts}
\usepackage[english]{babel}
\usepackage{amsmath}
\usepackage{amsthm}
\usepackage{bm}
\usepackage{url}

\usepackage{enumitem}

\usepackage{placeins}

\usepackage[usenames,dvipsnames,svgnames,table]{xcolor}

\usepackage{subfig}

\usepackage{multirow}

\newtheorem{theorem}{Theorem}[section]
\newtheorem{corollary}[theorem]{Corollary}

\newtheorem{proposition}[theorem]{Proposition}
\theoremstyle{definition}
\newtheorem{example}{Example}[section]
\theoremstyle{definition}
\newtheorem{remark}{Remark}[section]
\theoremstyle{definition}
\newtheorem{definition}{Definition}[section]

\usepackage{mathtools}
\newcommand{\defeq}{\vcentcolon=}

\usepackage[running]{lineno}

\newcommand*\patchAmsMathEnvironmentForLineno[1]{%
  \expandafter\let\csname old#1\expandafter\endcsname\csname #1\endcsname
  \expandafter\let\csname oldend#1\expandafter\endcsname\csname end#1\endcsname
  \renewenvironment{#1}%
     {\linenomath\csname old#1\endcsname}%
     {\csname oldend#1\endcsname\endlinenomath}}%
\newcommand*\patchBothAmsMathEnvironmentsForLineno[1]{%
  \patchAmsMathEnvironmentForLineno{#1}%
  \patchAmsMathEnvironmentForLineno{#1*}}%
\AtBeginDocument{%
\patchBothAmsMathEnvironmentsForLineno{equation}%
\patchBothAmsMathEnvironmentsForLineno{align}%
\patchBothAmsMathEnvironmentsForLineno{flalign}%
\patchBothAmsMathEnvironmentsForLineno{alignat}%
\patchBothAmsMathEnvironmentsForLineno{gather}%
\patchBothAmsMathEnvironmentsForLineno{multline}%
}

\numberwithin{equation}{section}
\numberwithin{table}{section}
\numberwithin{figure}{section}

\title{Boundary Treatment and Multigrid Preconditioning for Semi-Lagrangian Schemes Applied to Hamilton-Jacobi-Bellman Equations}

\author{
C.~Reisinger\thanks{Mathematical Institute, University of Oxford, Oxford, OX2 6GG, {\tt reisinge@maths.ox.ac.uk}.}
~and J. Rotaetxe Arto\thanks{Mathematical Institute, University of Oxford, Oxford, OX2 6GG, {\tt rotaetxe@maths.ox.ac.uk}. This author was partially supported by the \textit{Programa de Formaci\'{o}n de Investigadores del DEUI del Gobierno Vasco}.}
}

\begin{document}

\maketitle

\begin{abstract}
We analyse two practical aspects that arise in the numerical solution of Hamilton-Jacobi-Bellman (HJB) equations by a particular class of monotone approximation schemes known as semi-Lagrangian schemes. These schemes make use of a wide stencil to achieve convergence and result in discretization matrices that are less sparse and less local than those coming from standard finite difference schemes. This leads to computational difficulties not encountered there.
In particular, we consider the overstepping of the domain boundary and analyse the accuracy and stability of stencil truncation.
This truncation imposes a stricter CFL condition for explicit schemes in the vicinity of boundaries than in the interior, such that implicit schemes become attractive.
We  then study the use of geometric, algebraic and aggregation-based multigrid preconditioners to solve the resulting discretised systems from implicit time stepping schemes efficiently. Finally, we illustrate the performance of these techniques numerically for benchmark test cases from the literature. 
\end{abstract}



\section{Introduction}


We consider semi-Lagrangian schemes, as described in \cite{camilli1995approximation, debrabant2013semi}, for the numerical approximation of solutions to the 
Hamilton-Jacobi-Bellman (HJB) equation
\begin{align}
\label{eq:1.1} u_t - \inf_{\alpha \in \mathcal{A}} \left\lbrace L^{\alpha}[u](t, x) + c^{\alpha}(t, x) u(t,x) + f^{\alpha}(t, x) \right\rbrace &= 0, & (t,x) \in 
(0, T] \times \Omega, \\
\label{eq:1.2}  u(0, x) &= g(x), 
& x \in \bar \Omega,  \\
\label{eq:1.3} u(t, x) &= \psi(x), 
& (t,x) \in (0, T] \times \partial \Omega,
\end{align}
where $\Omega$ is a domain, $Q_T \defeq (0, T] \times \bar{\Omega}$ with $\bar{\Omega} \defeq \Omega \cup \partial \Omega \subseteq \mathbb{R}^d$, $\mathcal{A}$ is a compact set,
\begin{align} \label{eq:def_L}
L^{\alpha}[u](t, x) = \text{tr}[a^{\alpha}(t, x) D^2u(t, x)]	+ b^{\alpha}(t,x)Du(t,x)
\end{align}
is a second order differential operator, and $\psi$ and $g$ 
are the Dirichlet and initial conditions.

The coefficients $a^{\alpha} = \frac{1}{2} \sigma^{\alpha} \sigma^{\alpha, T}$, $b^{\alpha}$, $c^{\alpha}$, $f^{\alpha}$, the initial data $g$ and the boundary conditions $\psi$ take their values, respectively, in $\mathbb{S}^d$, the space of $d \times d$ symmetric matrices, $\mathbb{R}^d$, $\mathbb{R}$, $\mathbb{R}$, $\mathbb{R}$, and $\mathbb{R}$, and $\sigma^{\alpha} \in \mathbb{R}^{d \times P}$ such that $a^{\alpha}$ is positive semi-definite. 
We also assume the usual well-posedness conditions on the PDE coefficients, i.e.\ Lipschitz continuous in $x$ uniformly in $\alpha$, Hölder continuous with exponent $\frac{1}{2}$ in time and continuous in $\alpha$ for each $(t, x) \in Q_T$ \cite{LionsControl2}.
The relevant notion of solution for this type of non-linear equations is that of viscosity solutions \cite{crandall1992user} and the above conditions guarantee existence and uniqueness.

In general, the viscosity solution to \eqref{eq:1.1}--\eqref{eq:1.3} is unknown, thus it is necessary in practice to  compute approximations numerically. 
Sufficient conditions for a numerical scheme to converge to the unique viscosity solution of \eqref{eq:1.1}--\eqref{eq:1.3} were proved by Barles and Souganidis \cite{barles1991convergence} in terms of consistency, $L^\infty$-stability and monotonicity. We restrict our attention to finite difference discretizations of the differential operator \eqref{eq:def_L}. 

The requirement of 
monotonicity drastically affects the properties and construction of finite difference schemes. Theorem 4 in \cite{oberman2006convergent} proves that local monotone discretizations have at most first order for first-order equations and second order for second-order equations. 
What is more, standard fixed stencil methods are monotone only under restrictions on the diffusion matrix, such as diagonal dominance \cite{debrabant2013semi, forsythVetzal2012}. 
Results from \cite{crandallLionsConvergentSchemes, MotzkinWasow} further illustrate the limitations of 
such methods for the monotone approximation of second order derivatives.

This implies that generally approximations have to be non-local on the discrete level, i.e.\ the distance between mesh points involved in the scheme at a given point grows in relation to the mesh width as the mesh is refined.
Such schemes are referred to as wide stencils. 
For general diffusion matrices, first order accurate wide stencils of the type considered here have been proposed in \cite{camilli1995approximation, debrabant2013semi}, and a mixed fixed- and wide-stencil scheme in \cite{ma2014unconditionally}.

In this article, we analyse two issues arising in practice when numerically solving \eqref{eq:1.1}--\eqref{eq:1.3} using the class of schemes described in \cite{menaldi1989some, camilli1995approximation, debrabant2013semi} to discretize the second order differential operator \eqref{eq:def_L}. This approximation combines wide stencils in the directions determined by the columns of the diffusion matrix $\sigma^\alpha$ and the drift $b^\alpha$, together with (linear) interpolation. 
Following the notation in \cite{debrabant2013semi}, we write the matrix $\sigma^\alpha \in \mathbb{R}^{d \times P}$ as $(\sigma^\alpha_1, \sigma^\alpha_2, \ldots, \sigma^\alpha_P)$, where $\sigma^\alpha_p \in \mathbb{R}^d$ for $p \in \{1, 2, \ldots, P\}$ denotes the $p$-th column of $\sigma^\alpha$, and observe that for $k > 0$ and any smooth function $\phi$,
\begin{align}
\label{eq:D2approx}
\frac{1}{2} \text{tr} \left[ \sigma^\alpha \sigma^{\alpha ~ T} D^2 \phi(x) \right] &= \frac{1}{2} \sum_{p = 1}^P  \frac{\phi (x + k \sigma^\alpha_p ) -2 \phi(x) + \phi (x - k \sigma^\alpha_p) }{k^2} + \mathcal{O}(k^2), \\
\label{eq:D1approx}
b^\alpha D\phi(x) &= \frac{\phi(x + k^2 b^\alpha ) - \phi (x)}{k^2} + \mathcal{O}(k^2),
\end{align}
where $\mathcal{O}(k^2)$ is the local truncation error of the finite difference and for compactness we write $b^\alpha \equiv b^\alpha(t, x)$ and $\sigma^\alpha \equiv \sigma^\alpha(t, x)$. As these approximations will be used for points lying on a discrete spatial grid $\Omega_{\Delta x}$
with nodes $\{x_j: 1 \le j \le N\}$, the displaced points $x + k^2 b^\alpha$, $x \pm k \sigma^\alpha_p$ do not generally coincide with nodes of $\Omega_{\Delta x}$. Therefore, $\phi$ is replaced by an interpolant $\mathcal{I}_{\Delta x} \phi$ on that grid. We restrict our attention to linear interpolants, defined by the standard piecewise multilinear non-negative basis functions $\{w_j(\cdot): 1 \le j \le N\}$ associated with the mesh nodes, such that for any function $\phi$ 
\begin{align} \label{eq:monotone_int}
(\mathcal{I}_{\Delta x} \phi) (x) = \sum_{j\in \mathcal{N}(x)} \phi(x_j) w_j(x),
\end{align}
for all $x \in \Omega$, $x_j \in \Omega_{\Delta x}$, where 
$\mathcal{N}(x)$ is the set of neighbours of $x$ on the mesh $\Omega_{\Delta x}$, i.e.\ the mesh points with non-zero interpolation weight.
The resulting scheme is referred to as the Linear Interpolation Semi-Lagrangian (LISL) scheme.

It is shown in \cite{debrabant2013semi} that the leading order terms of the local truncation error are proportional to $k^2$ and $\frac{\Delta x^2}{k^2}$, where the last quantity corresponds to the linear interpolation error in the finite difference formulae (\ref{eq:D2approx}) and (\ref{eq:D1approx}) by replacing $\phi$ by its interpolant. Therefore, by choosing $k = \sqrt{\Delta x}$, the resulting scheme is locally of first order in $\Delta x$.


Following the notation in \cite{debrabant2013semi}, the LISL finite difference approximations for the differential operator in \eqref{eq:def_L} can be expressed as 
\begin{align} \label{eq:FDscheme}
L_{\Delta x}^{\alpha}[\mathcal{I}_{\Delta x}\phi](t, x) \defeq \sum_{p =1}^M \frac{(\mathcal{I}_{\Delta x}\phi)(t, x + y_{p}^{\alpha, +}(t, x)) - 2 (\mathcal{I}_{\Delta x}\phi)(t,x) + (\mathcal{I}_{\Delta x}\phi)(t, x + y_{p}^{\alpha, -}(t, x))}{2 \Delta	x},
\end{align}
for $x \in \Omega_{\Delta x}$, and some $M \geq 1$. 

Different schemes can be obtained depending on the values taken by $M$ and $y_{p}^{\alpha, \pm}(t, x)$. 
In particular, \cite{debrabant2013semi} discusses the following three schemes:
\paragraph{Examples of LISL schemes.}
\label{sec:LISLexamples}
\begin{enumerate}
\item 
{\bf Scheme 1}: The approximation of Camilli and Falcone \cite{camilli1995approximation}, corresponding to $y^{\alpha, \pm}_{p} =  \pm \sqrt{\Delta x} \sigma^{\alpha}_{p} + \frac{\Delta x}{P} b^{\alpha}$ and $M = P$.
\item 
{\bf Scheme 2}: The approximation in \cite{debrabant2013semi}, corresponding to $y^{\alpha, \pm}_{p} =  \pm \sqrt{\Delta x} \sigma^{\alpha}_{p}$ for $p \leq P$, $y^{\alpha, \pm}_{P+1} = \Delta x b^{\alpha}$, and $M = P + 1$.
\item 
{\bf Scheme 3}:
A more efficient version of the Camilli-Falcone approximation, corresponding to $y^{\alpha, \pm}_{p} =  \pm \sqrt{\Delta x} \sigma^{\alpha}_{p}$ for $p < P$, $y^{\alpha, \pm}_{P} =  \pm \sqrt{\Delta x} \sigma^{\alpha}_{P} + \Delta x b^{\alpha}$, and $M = P$.
\end{enumerate}


The authors show that this family of discretizations of \eqref{eq:def_L} is consistent and monotone. Monotonicity of the scheme is fulfilled as the discrete approximation $L_{\Delta x}^{\alpha}\left[ \mathcal{I}_{\Delta x} \phi \right]$ is the composition of monotone finite differences and a monotone interpolation operation. Once discretized in space, the final scheme arises from discretising in time using the standard $\theta$-time stepping scheme for $\theta \in [0, 1]$, where $\theta = 0$ corresponds to the explicit Euler time stepping and $\theta = 1$ to the implicit case, on a time grid represented by a strictly increasing sequence of points $\{t_n\}_{n = 0}^{N_t+ 1}$ with $t_0 = 0$, $t_{N_t+1} = T$, and $\Delta t_n \defeq t_n - t_{n-1} \leq \Delta t$ for all $n$. The scheme being monotone, it can be written as described in the following definition, where for any grid function $V:\{t_n\}_{n = 0}^{N_t+ 1} \times \Omega_{\Delta x} \to \mathbb{R}$, $V^n_i \equiv V(t_n, x_i)$.

\begin{definition}[Equation (4.1) in \cite{debrabant2013semi}] 
\label{def:positiveType}
A scheme is said to be of positive type, if it can be written as
\begin{align} \label{eq:positiveScheme}
\max_{\alpha \in \mathcal{A}} \left\lbrace  \mathcal{B}^{\alpha, n,n}_{j, j} U^n_j - 
\sum_{ i \neq j}
\mathcal{B}^{\alpha, n,n}_{j, i} U^n_i - \sum_{i=1}^N  \mathcal{B}^{\alpha, n,n-1}_{j, i} U^{n-1}_i  - F^{\alpha, n - 1 + \theta}_j   \right\rbrace = 0,
\end{align}
for $ j = 1,\ldots, N$, on the discrete domain $\{t_n\}_{n = 0}^{N_t+ 1} \times \Omega_{\Delta x}$, where $U^n_i$ is the numerical solution at node $(t_n, x_i)$ and all the coefficients $\mathcal{B}$ are non-negative.
\end{definition}

For the convenience of the reader, we reproduce the expressions for $\mathcal{B}^{\alpha, n, \cdot}_{ j, \cdot}$ of the LISL schemes as in \cite{debrabant2013semi}, for all $1\le i \neq j \le N$, $x_i,x_j \notin \partial \Omega$,
\begin{align*}
&\mathcal{B}^{\alpha, n, n}_{j, j} = 1 + \theta  \Delta t_n \left( \frac{M}{2 \Delta x} - l^{\alpha, n}_{j, j} - c^{\alpha, n-1+\theta}_j \right), 
&& \mathcal{B}^{\alpha, n, n}_{j, i} = \theta  \Delta t_n  \, l^{\alpha, n}_{j, i},
\\
&\mathcal{B}^{\alpha, n, n-1}_{j, j} = 1 - (1 - \theta) \Delta t_n \left( \frac{M}{2 \Delta x} - l^{\alpha, n-1}_{j, j} - c^{\alpha, n-1+\theta}_j \right), 
&&\mathcal{B}^{\alpha, n, n-1}_{j, i} = (1-\theta)  \Delta t_n \, l^{\alpha, n-1}_{j, i},
\end{align*}
where $c^{\alpha, n-1+\theta}_j = c^\alpha(t_{n-1}+\theta \Delta t,x_j)$ and
\begin{align*}
l^{\alpha, n}_{ j, i} = \sum^{M}_{p=1} \frac{w_{i} (x_j + y^{\alpha, +}_{p}(t_{n}, x_j)) +  w_{i} (x_j + y^{\alpha, -}_{p} (t_{n}, x_j))}{2 \Delta x}.
\end{align*}



The schemes described above have a wide stencil as the length of the stencil, being proportional to the ratio $k/\Delta x \sim 1/\sqrt{\Delta x}$,
tends to $\infty$ as $\Delta x \to 0$. Hence, when applied on a bounded discrete grid, the stencil will generally exceed the domain for points close to its boundary. As discussed in \cite{debrabant2013semi}, the overstepping may pose a problem depending on the equation and the type of boundary conditions imposed. We consider Dirichlet boundary conditions here.

Our first goal is to present and analyse a modification of the LISL scheme to deal with overstepping for problems on bounded domains with Dirichlet boundary conditions, and general drift and diffusion coefficients. We describe how to truncate the LISL stencil so that the truncation remains consistent and monotone. 
We prove that the resulting stencil for Scheme 2 above is of positive type (as per Definition \ref{def:positiveType}), and since the coefficients $\mathcal{B}$ in (\ref{eq:positiveScheme}) do not depend on $U$, it is also monotone. This is not the case for Schemes 1 and 3. 
We also observe that the truncation has both local and global impacts on the properties of the scheme. 
Locally, the modification of the scheme leads to a loss of accuracy of half an order in the consistency error, i.e.\ $\mathcal{O}(\sqrt{\Delta x})$ instead of $\mathcal{O}(\Delta x)$, 
due to the loss of symmetry. 
We compare the accuracy of the truncation with extrapolations of the boundary conditions by way of numerical tests for benchmark problems.
As the mesh points requiring truncation of the scheme are restricted to an $\mathcal{O}(\sqrt{\Delta x})$ layer at the boundary, convergence rates close to $\mathcal{O}(\Delta x)$ are observed empirically for the new scheme. 
The truncation has a global effect in the sense that it modifies the CFL condition of explicit schemes by at least half an order, from $\Delta t = \mathcal{O}(\Delta x)$ to $\Delta t = \mathcal{O}(\Delta x^{3/2})$.
As the empirical error is $\mathcal{O}(\Delta t) + \mathcal{O}(\Delta x)$ for fully implicit schemes, the computationally most efficient choice is $\Delta t \sim \Delta x$,
outside the stability region of explicit schemes.

The second goal is therefore the use of implicit schemes and the efficient solution of the discrete system \eqref{eq:positiveScheme} using multigrid preconditioning. 
For $\theta \neq 0$, the coupling of the optimal control and the coefficients makes \eqref{eq:positiveScheme} a non-linear system of algebraic equations,
\begin{align} \label{eq:nonlinear_sys}
\max_{\alpha \in \mathcal{A}} \left(A_i^\alpha X - F_i^\alpha \right) = 0, \qquad i=1,\ldots, N, 
\end{align}
where $A_i^\alpha$ is the $i$-th row of a matrix $A^\alpha$ with elements $A_{i,j}^\alpha$, $i,j=1,\ldots, N$, and control $\alpha \in \mathcal{A}$. Comparing with (\ref{eq:positiveScheme}),
$A_{i,j}^\alpha = \mathcal{B}^{\alpha, n, n}_{i, j}$, $F_i^\alpha = F^{\alpha, n - 1 + \theta}_i$, and $X= (X_i) = (U^n_i)$ is the solution vector for the $n$-th time step. The maximisation over $\alpha$ in (\ref{eq:nonlinear_sys}) is row-wise and usually done by linear search.
By construction of the LISL scheme, $A^\alpha$ is an M-matrix with non-negative row sum. 
Therefore, following results in \cite{bokanowski_Howard}, we can use policy iteration to compute $U$. Then, within each policy iteration, a linear system
$A_i^{\alpha_i} X = F_i^{\alpha_i}$, $i=1,\ldots, N$, with fixed control vector $(\alpha_i)_{1\le i \le N}$ has to be solved. 
We find (in contrast to \cite{ma2014unconditionally}) that this last step is the computationally most costly part of the overall algorithm if direct linear solvers or standard iterative solvers are used.\footnote{Which of the two steps is more costly depends crucially on the type of control problem. The optimisation step is typically fast if the control is taken from a finite set, if the local control problem is analytically solvable (e.g., quadratic or of `bang-bang'-type), or if the coefficients are a smooth convex function of the control, such that standard Newton-type methods can be used. It will be more costly, if the optimal control has to be approximated by exhaustive search over a discretised control set, especially if the dimension of the control space is higher than the spatial dimension of the PDE. In the examples considered in this paper, the control is scalar and we optimise by linear search over a one-dimensional control mesh.
}
We therefore study multigrid preconditioners. (See Table \ref{Tab:solverpercentage}.)

In the literature on multigrid for HJB equations, two main approaches are observed: on the one hand, multigrid is applied directly to the non-linear problem, as in \cite{BlobetaHoppe89, hanWanMGHJBI, hoppeMGHJB}; and on the other hand, multigrid is applied to a linearised problem, as in \cite{akian1995finite}. 
In particular, \cite{BlobetaHoppe89, hoppeMGHJB} provide the first multigrid algorithms for HJB equations and prove convergence,
while
\cite{hanWanMGHJBI} presents a novel smoother for HJB equations based on damped value iteration \cite{kushner2001numerical}.
These articles have in common the use of standard fixed stencil finite difference approximations 
and the use of a geometric structure when building the hierarchy of multigrid subspaces. 

The novelty of this article is to study the application of multigrid preconditioning to a wide stencil discretization.
We will demonstrate, both by Fourier analysis of a model problem and by numerical tests in a more complex application, that standard geometric multigrid does not give mesh-size independent convergence.

We then investigate algebraic multigrid methods.
The basis for the specific algorithm we use was introduced in \cite{notay2010aggregation} for linear elliptic PDEs. It empirically showed that ``aggregation based methods could yield robust\footnote{In this context, a robust method is referred to as one showing good performance for a large range of problems without changing the smoother.} and convergent schemes if used as preconditioners of a Krylov method, and were part of an enhanced multigrid cycle, not simple V- or W-cycles" as considered in \cite{stubenAMGreview}. By enhanced multigrid cycles, the authors refer to recursive schemes in which at each coarse level the solution to the residual equation is computed using a number of Krylov subspace iterations as in \cite{notay2008krylov} or with a semi-iterative method based on Chebyshev polynomials called the AMLI cycle, see Section 5.6 of \cite{vassilevski2008multilevel}. The aggregates were formed using heuristic criteria following coupling in the strongest direction.

In \cite{napov2012AGMGguaranteed} the authors introduced an aggregation-based multigrid method with guaranteed convergence rate for symmetric M-matrices with non-negative row sum. 
A LISL discretization matrix is only symmetric in very specific cases with limited practical interest. For non-symmetric matrices, in \cite{notay2012aggregation} convergence of a simplified two-grid scheme using aggregation is proved for non-singular M-matrices with non-negative row and column sums. This requirement ensures that the symmetric part of the coefficient matrix $A$ given by $A + A^T$ meets the assumptions in \cite{napov2012AGMGguaranteed} and allows the use of its theoretically justified algorithms. 
We will derive conditions on the coefficients of the HJB equation such that this theory applies, and
show empirically that aggregation-based multigrid gives roughly mesh-size independent convergence. 

The rest of the article is organised as follows. Section \ref{sec:Boundary} discusses the truncation of the LISL scheme for points whose stencil exceeds the domain and compares its performance to na\"{\i}ve extrapolations of the boundary conditions. Section \ref{sec:MG} considers the application of three different multigrid methods to linear systems where the coefficient matrix arises from LISL discretizations. Section \ref{sec:End} contains the final remarks.

\section{Boundary treatment for the LISL scheme}
\label{sec:Boundary}

In this section, we analyse adaptations of Schemes 1--3 for initial-boundary value problems on bounded domains.
As described in the introduction, for points $x$ close to the boundaries of the domain, the stencil points $x + y^{\alpha, \pm}_{p} (t, x)$ in (\ref{eq:FDscheme}) generally do not lie in a mesh element. 
In the following, we therefore discuss the truncation of \eqref{eq:FDscheme} so that the resulting scheme remains monotone, consistent, and $L^{\infty}$-stable. The proposed truncation samples the boundary points on the straight lines defined by the point $x$ and $x + y^{\alpha, \pm}_{p}(t, x)$ and adjusts the corresponding finite difference weights for consistency. 

\subsection{Definition of truncated stencils}

We take $\Omega \subset \mathbb{R}^d$ for $d \geq 2$. 
We first outline how the method can be defined on a general domain with curved boundary, but later (especially in the numerical tests) focus for simplicity on rectangular domains.
We start with a Cartesian mesh on $\mathbb{R}^d$ with uniform mesh width $\Delta x$ and then choose $\Omega_{\Delta x}$ as all the points which lie inside $\Omega$.
See Figure \ref{fig:domain}.
\begin{figure}[htp]
  \centering
	\includegraphics[width=0.5\textwidth]{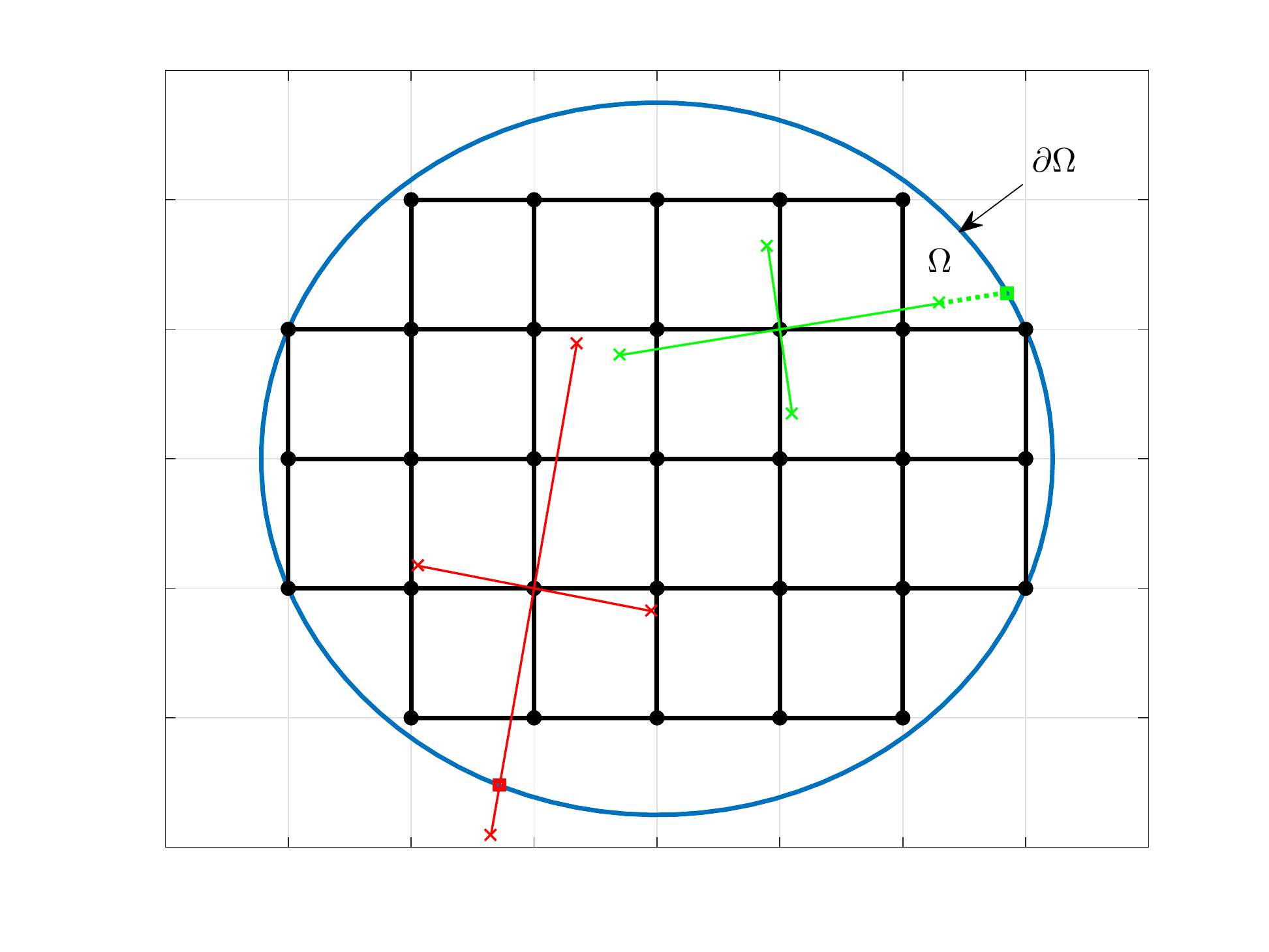}
  \caption{Truncation and extrapolation of the stencil for an elliptical domain and a mesh made of square cells. The modified stencil samples the domain boundary.} 
           \label{fig:domain}
\end{figure}

We now fix a mesh node $x \in \Omega_{\Delta x}$ .
There are two distinct situations where interpolation at the point $x + y^{\alpha, \pm}_{p}(t, x)$ as per (\ref{eq:FDscheme}) is not possible  for given $t, \alpha$ and $p$:
\begin{itemize}
\item[A.]
$x + y^{\alpha, \pm}_{p}(t, x) \notin \bar{\Omega}$ (bottom left in Fig.~\ref{fig:domain});
\item[B.]
$x + y^{\alpha, \pm}_{p}(t, x) \in \bar{\Omega}$, but the element it is contained in has vertices outside $\bar{\Omega}$ (top right).
\end{itemize}
We say the stencil ``oversteps''.
In such cases,
the objective is to find truncated or extended stencil vectors $\hat{y}^{\alpha, \pm}_{p}(t, x)$ and corresponding finite difference weights $A^\alpha_p\equiv A^\alpha_p(t, x)$ and $B^\alpha_p \equiv B^\alpha_p(t, x)$, such that $x + \hat{y}^{\alpha, \pm}_{p}(t, x) \in \partial \Omega$ and the truncated scheme
\begin{align} \label{eq:TruncScheme}
& \hat{L}_{\Delta x}^{\alpha}[\mathcal{I}_{\Delta x}\phi](t, x)  \defeq \nonumber \\
& \quad \sum_{p =1}^M \frac{A^\alpha_p (\mathcal{I}_{\Delta x}\phi)(t, x + \hat{y}_{p}^{\alpha, +}(t, x)) - (A^\alpha_p + B^\alpha_p) \phi(t, x) + B^\alpha_p (\mathcal{I}_{\Delta x}\phi)(t, x + \hat{y}_{p}^{\alpha, -}(t, x)) }{2 \Delta x}
\end{align}
is a consistent approximation of \eqref{eq:def_L} as $\Delta x \to 0$. If the stencil does not overstep, we have that $\hat{y}_{p}^{\alpha, \pm}(t, x) = y_{p}^{\alpha, \pm}(t, x)$ and $A^\alpha_p = B^\alpha_p = 1$. If it does, for any $t$ we define
\[
\hat{y}^{\alpha, \pm}_{p}(t, x) = \mu^{\alpha, \pm}_p(t, x) y_{p}^{\alpha, \pm}(t, x), \quad \text{ where } \quad \mu^{\alpha, \pm}_p(t, x) = \min \left \{ \mu \ge 0 \,: \,
x + \mu y_{p}^{\alpha, \pm}(t, x) \in \partial \Omega\right \}.
\]
In case {\rm A}, this means $\mu <1$, while in case {\rm B} we have $\mu >1$.


In the remainder of this section we restrict our attention to the truncation of the scheme on rectangular domains, in which case the elements of the Cartesian mesh cover exactly the domain and case {\rm B} does not occur.
Moreover, this means that interior mesh points cannot be arbitrarily close to the boundary, but are always at least $\Delta x$ away\footnote{This
can also be enforced in the general case by removing the outermost layer of cells, such that again a distance of $\Delta x$ between non-boundary mesh points and the domain boundary is ensured.}. This allows the derivation of CFL conditions for the explicit schemes as given below in Section \ref{subsec:properties}.


\subsection{Consistency conditions}
\label{sec:Consistency}

In the truncated scheme (\ref{eq:TruncScheme}) there are $M$ pairs of weights, which can be chosen freely, subject to positivity, in order to obtain a consistent scheme.
As we will see below, this is only possible for Scheme 2. 

In the following, we denote $[[1, j]] \equiv [1, j] \cap \mathbb{Z}$ and for a vector $v \in \mathbb{R}^d$, $(v)_i$ denotes its $i$-th element. 
As in the introduction, we have that  $b^{\alpha} \in \mathbb{R}^{d}$, and $\sigma^{\alpha} = (\sigma^{\alpha}_1, \ldots, \sigma^{\alpha}_p, \ldots, \sigma^{\alpha}_P) \in \mathbb{R}^{d \times P}$ where $\sigma^{\alpha}_p \in \mathbb{R}^d$ denotes the $p$-th column vector. 
For compactness, we omit the dependence of the coefficients and the stencil related functions with respect to the position, that is $b^{\alpha} \equiv b^{\alpha}(t, x)$, $\sigma^{\alpha}_p \equiv \sigma^{\alpha}_p(t, x)$, $y_{p}^{\alpha, \pm} \equiv y_{p}^{\alpha, \pm}(t, x)$ and $\mu_{p}^{\alpha, \pm} \equiv \mu_{p}^{\alpha, \pm}(t, x)$. 
We add a second subscript taking values 1, 2 or 3 to $A^\alpha_{p}$, $B^\alpha_{p}$ and $y_{p}^{\alpha, \pm}$ to make the discretization scheme explicit.

%

\begin{proposition}
The truncated version of Schemes 1 and 3 is generally not consistent.
\end{proposition}

\begin{proof}
By Taylor expansion of a smooth test function we find that the consistency conditions for Scheme 1 are
\begin{align*}
\sum_{p \in \mathcal{P}} \left( A^\alpha_{1, p} (\hat{y}^{\alpha, +} _{1, p})_i + B^\alpha_{1, p} (\hat{y}^{\alpha, -}_{1, p})_i \right)
	&= 2 \Delta x \frac{|\mathcal{P}|}{P} (b^{\alpha})_i + o(\Delta x), \\
\sum_{p \in \mathcal{P}} \left( A^\alpha_{1, p} (\hat{y}^{\alpha, +}_{1, p})_{i_1} (\hat{y}^{\alpha, +}_{1, p})_{i_2} + B^ \alpha_{1, p} (\hat{y}^{\alpha, -}_{1, p})_{i_1} (\hat{y}^{\alpha, -}_{1, p})_{i_2} \right)
	&= 2\Delta x \sum_{p \in \mathcal{P}} (\sigma^{\alpha}_{p})_{i_1} (\sigma^{\alpha}_{p})_{i_2} + o(\Delta x),
\end{align*}
where $\mathcal{P} \subseteq [[1, P]]$ denotes the set of stencils overstepping the domain and $i, i_1, i_2 \in [[1, d]]$. 

In Scheme 1, there are $2 |\mathcal{P}| \le 2 d$ variables, but $(d^2 + 3d)/2$ 
equations, $d$ from the condition on the Jacobian and $(d^2 +d)/2$ from the condition on the Hessian.
This overdetermined system has a solution only if there is linear dependence between the equations. Except for special cases, e.g.\ $|\mathcal{P}|=0$ or $\sigma^\alpha_p$ parallel to $b^\alpha$ for some $p$, this is not the case. Hence, in general the truncated Scheme 1 is not consistent.


We observe that the same principle applies to Scheme 3 for 
$y^{\alpha, \pm}_{3, P} = \pm \sqrt{\Delta x} \sigma_P^{\alpha} + \Delta x b^{\alpha}$.

\end{proof}

For example, consider $x_0 = (0, 0)^T$, $\bar{\Omega} = [-5, 1]^2$, $\sqrt{\Delta x} \sigma^\alpha_1(x_0) = (2, 0)^T$, $\sqrt{\Delta x} \sigma_2^\alpha(x_0) = (0, 1)^T$, and $\Delta x b^\alpha(x_0) = (0, 1)^T$, then the truncated version of Scheme 1 is not consistent, but the one for Scheme 3 is. 
However, if $\Delta x b^\alpha(x_0) = ( 1, 1)^T$ then neither of them is consistent.

We conclude that for points whose stencil oversteps the boundary, the approximations of the first and second derivative should be considered separately, as done in Scheme 2.

\begin{proposition}
\label{pro:scheme2}
For Scheme 2 
and all $p \in [[1, P+1]]$, let 
$\mu^{\alpha, \pm}_{p} \in (0, 1]$ be the largest constant such that 
$x + \mu \ y_{2,p}^{\alpha, \pm} \in \bar{\Omega}$ for all $\mu \in [0,\mu^{\alpha, \pm}_{p}]$, and define
\begin{align}
\label{eqn_APp1}
A^\alpha_{2, P+1} = B^\alpha_{2, P+1} = \frac{1}{\mu^{\alpha, +}_{P+1}} \left(= \frac{1}{\mu^{\alpha, -}_{P+1}}\right),
\end{align}
and, for $p \in [[1, P]]$,
\begin{align}
\label{eqn_Ap}
A^\alpha_{2, p}  = \frac{2}{(\mu^{\alpha, +}_{p})^2 + \mu^{\alpha, +}_{p} \mu^{\alpha, -}_{p}}, \qquad B^\alpha_{2, p} = \;\;  \frac{2}{(\mu^{\alpha, -}_{p})^2 + \mu^{\alpha, -}_{p} \mu^{\alpha, +}_{p}}.
\end{align}

Then the scheme defined by (\ref{eq:TruncScheme}) is consistent unless both $\mu^{\alpha, +}_{p}, \mu^{\alpha, -}_{p} \sim O(\sqrt{\Delta x})$.
\end{proposition}

\begin{proof}
If the stencil oversteps, then the truncated stencil consists of the point at the intersection between the boundary $\partial \Omega$ and one of the segments $\{x, x+ \sqrt{\Delta x} \sigma^{\alpha}_p\}$, $\{x, x - \sqrt{\Delta x} \sigma^{\alpha}_p\}$, or $\{x, x+ \Delta x b^{\alpha}\}$. 
For each point $(t, x)$ Scheme 2 requires the calculation of at most $2P + 1$ different weights, i.e.\ $2P$ for the second order term and one for the first order term. 
For the latter we have that $\hat{y}^{\alpha, +}_{2, P+1} = \hat{y}^{\alpha, -}_{2, P+1}$, therefore $A^\alpha_{2, P+1}= B^\alpha_{2, P+1}$. 
Ignoring the interpolation error for the time being, the coefficients are obtained from the consistency conditions (up to a term $o(\Delta x)$), 
\begin{align}
\label{eq:AB_drift} (A^\alpha_{2, P+1} + B^\alpha_{2, P+1}) (\hat{y}^{\alpha, \pm}_{2, P+1})_i &= 2 \Delta x (b^{\alpha})_i, &  &\forall i \in [[1, d]],
\end{align}
for the first order term, and 
\begin{align}
\label{eq:AB1}A^{\alpha}_{2, p} (\hat{y}^{\alpha, +}_{2,p})_i + B^{\alpha}_{2, p} (\hat{y}^{\alpha, -}_{2, p})_i &= 0, &  &\forall i \in [[1, d]], \\
\label{eq:AB2}A^{\alpha}_{2, p} (\hat{y}^{\alpha, +}_{2, p})_{i_1} (\hat{y}^{\alpha, +}_{2, p})_{i_2} + B^{\alpha}_{2, p} (\hat{y}^{\alpha, -}_{2, p})_{i_1} (\hat{y}^{\alpha, -}_{2, p})_{i_2} &= 2 \Delta x (\sigma^{\alpha}_{p})_{i_1} (\sigma^{\alpha}_{p})_{i_2}, &  &\forall (i_1, i_2) \in [[1, d]]^2,
\end{align}
for the second order term.

By construction of the truncated stencil \eqref{eq:AB_drift} and \eqref{eq:AB1} are linearly dependent across $i$, and \eqref{eq:AB2} across $i_1$ and $i_2$, resulting in one (linearly independent) equation for the first order term weights and two for $A^\alpha_{2, p}$, $B^\alpha_{2, p}$,
with solutions given by
\begin{align} \label{eq:DriftCoeff}
A^\alpha_{2, P+1} = B^\alpha_{2, P+1} = \Delta x \frac{(b^{\alpha})_i}{(\hat{y}^{\alpha, \pm}_{2, P+1})_i},
\end{align}
and
\begin{align} \label{eq:ABcoeff}
A^\alpha_{2, p} =  \frac{2 \Delta x (\sigma^{\alpha}_{p})^2_{i} }{(\hat{y}^{\alpha, +}_{2, p})_{i} ((\hat{y}^{\alpha, +}_{2, p})_{i} - (\hat{y}^{\alpha, -}_{2, p})_{i})}, \qquad B^\alpha_{2, p} = \frac{2 \Delta x (\sigma^{\alpha}_{p})^2_{i} }{(\hat{y}^{\alpha, -}_{2, p})_{i} ((\hat{y}^{\alpha, -}_{2, p})_{i} - (\hat{y}^{\alpha, +}_{2, p})_{i})},
\end{align}
which are seen to be equivalent to equations (\ref{eqn_APp1}) and (\ref{eqn_Ap}).

The contribution to the consistency error of (\ref{eq:TruncScheme}) from the bilinear interpolation operator $\mathcal{I}$ is bounded by $(\Delta x)^{-1} \sum_p (|A_p|+|B_p|) (\Delta x)^2$, which is goes to 0 if $|A_p|+|B_p| = o((\Delta x)^{-1})$ for all $p$, which is violated if and only if $\mu^{\alpha, +}_{p}, \mu^{\alpha, -}_{p} \sim O(\sqrt{\Delta x})$.
\end{proof}

\begin{corollary} \label{cor:Growth}
For the truncated Scheme 2, (\ref{eq:TruncScheme}), (\ref{eqn_APp1}) and (\ref{eqn_Ap}), the following holds:
\begin{enumerate}[label=\alph*)]
\item
The scheme is of positive type and monotone with $A^{\alpha}_{2, p}, B^{\alpha}_{2, p} \geq 1$ for all $p \in [[1, P+1]]$.
\item
For points $x$ within a distance $\mathcal{O}(\Delta x)$ of the boundary and $p \neq P+1$,
as $\Delta x \to 0$, 
\begin{align*}
\text{if } |\hat{y}^{\alpha, +}_{2,p}| < \sqrt{\Delta x}|\sigma^{\alpha}_{p}|  \text{ and } |\hat{y}^{\alpha, -}_{2,p}| =  \sqrt{\Delta x}|\sigma^{\alpha}_{p}| 
	&\implies A^{\alpha}_{2, p} \sim \mathcal{O}(\Delta x^{-1/2}) \text{ and } \lim_{\Delta x \to 0} B^{\alpha}_{2, p} = 2, \\
\text{if } |\hat{y}^{\alpha, -}_{2,p}| < \sqrt{\Delta x}|\sigma^{\alpha}_{p}| \text{ and } |\hat{y}^{\alpha, +}_{2,p}| = \sqrt{\Delta x}|\sigma^{\alpha}_{p}| 
	&\implies \lim_{\Delta x \to 0} A^{\alpha}_{2, p} = 2 \text{ and } B^{\alpha}_{2, p} \sim \mathcal{O}(\Delta x^{-1/2}), \\
\text{if } |\hat{y}^{\alpha, \pm}_{2,p}| < \sqrt{\Delta x}|\sigma^{\alpha}_{p}| 
	& \implies A^{\alpha}_{2, p}, B^{\alpha}_{2, p} \sim \mathcal{O}(\Delta x^{-1}).
\end{align*} 
\item
The local consistency error for points with truncation and $p \neq P+1$ is $\mathcal{O}(\sqrt{\Delta x})$ if only one side of the stencil oversteps, and $\mathcal{O}(1)$ if both sides overstep. 
\end{enumerate}
\end{corollary}

\begin{proof}
The claim in \textit{a)} follows from (\ref{eqn_APp1}), (\ref{eqn_Ap}), and the fact that $\mu^{\alpha, \pm}_{p} \in (0, 1]$ and the coefficients $A^{\alpha}_{2, p}, B^{\alpha}_{2, p}$ do not depend on the numerical solution $U$.
The limits in \textit{b)} follow from (\ref{eqn_Ap}) and noting that if the stencil oversteps for a point $x$ lying $\mathcal{O}(\Delta x)$ away from the boundary, but at least $\Delta x$ by the assumption made on the mesh, then $\mu^{\alpha, +}_{p} \sim \mathcal{O}(\sqrt{\Delta x})$ and/or $\mu^{\alpha, -}_{p} \sim \mathcal{O}(\sqrt{\Delta x})$, but not $o(\sqrt{\Delta x})$.

To prove \textit{c)} we use Taylor expansions for each $p$ and conclude using the limits in \textit{b)}. 
Let $\phi : \bar{\Omega} \to \mathbb{R}$ be a smooth function and for any $p \in (\mathcal{P} \cap [[1, P]])$, where $\mathcal{P}$ denotes the set of stencils overstepping the domain, then by Taylor expansion and the consistency conditions \eqref{eq:AB1}--\eqref{eq:AB2} the local consistency error $\tau$ for the $p$-th addend of \eqref{eq:TruncScheme} using multi-index notation is given by
\begin{align*}
\tau &\defeq \frac{A^\alpha_p \phi(t, x + \hat{y}_{p}^{\alpha, +}) - (A^\alpha_p + B^\alpha_p) \phi(t, x) + B^\alpha_p \phi(t, x + \hat{y}_{p}^{\alpha, -}) }{2 \Delta x} - \frac{1}{2} \mathrm{tr}[\sigma^\alpha_p \sigma^{\alpha, T}_p D^2 \phi] \\
&= \frac{1}{2 \Delta x} \sum_{|\beta| \geq 3} \frac{1}{|\beta|!} ( A^\alpha_p (\hat{y}_{p}^{\alpha, +})^\beta + B^\alpha_p (\hat{y}_{p}^{\alpha, -})^\beta ) D^\beta \phi,
\end{align*}
where, due to the truncation of the stencil, the scheme is not central and therefore the terms for odd $|\beta|$ do not cancel out.
If only one side of the stencil oversteps then for $|\beta| = 3$
$$\frac{A^\alpha_p (\hat{y}_{p}^{\alpha, +})^\beta + B^\alpha_p (\hat{y}_{p}^{\alpha, -})^\beta}{\Delta x} \sim \mathcal{O}( \sqrt{\Delta x}),$$
whereas if both sides overstep then
the error from interpolation dominates and is $\mathcal{O}(1)$ for points $\mathcal{O}(\Delta x)$ from the boundary, as seen at the end of the proof of Proposition \ref{pro:scheme2}.
\end{proof}

\begin{remark}[Two-sided overstepping] We note that it is possible for both sides of the stencil to overstep if the diffusion direction $\sigma^\alpha_p$ is (almost) parallel to the domain boundary, for points close to a locally convex smooth boundary with high curvature in that direction, as well as close to corners; see Remark \ref{rem:two-over} and Table \ref{Tab:shifted} below.

The scheme is consistent at points with two-sided overstepping if the truncated scheme is not interpolated at the boundary but uses the exact boundary values. In that case, the consistency error for those points is 
$\mathcal{O}( \Delta x)$.
\end{remark}

\subsection{Properties of the truncated stencil}
\label{subsec:properties}

The changes in the finite difference weights of scheme \eqref{eq:TruncScheme} introduced by the truncation, modify the positivity conditions given in Lemma 4.1 in \cite{debrabant2013semi}. 
We will show that the scheme remains conditionally $L_\infty$-stable and monotone, but the CFL conditions are more restrictive in the truncated case for time-stepping schemes with $\theta < 1$. We start by writing the scheme on a discrete time-space grid with mesh parameters $\Delta t$ and $\Delta x$ as
\begin{align}
&~\hat{L}^{\alpha}_{\Delta x} [\mathcal{I}_{\Delta x} \phi(t, \cdot)](t_{n}, x_j) \nonumber \\ 
&= \sum^{M}_{p=1} \frac{1}{2 \Delta x} \left[ A^{\alpha, n}_p (\mathcal{I}_{\Delta x} \phi(t_{n}, \cdot))(x_j + \hat{y}^{\alpha, +}_{p}) - (A^{\alpha, n}_p + B^{\alpha, n}_p) \phi(t_{n}, x_j) \right. \nonumber \\
&\qquad \qquad \qquad \left. + B^{\alpha, n}_p (\mathcal{I}_{\Delta x} \phi(t_{n}, \cdot))(x_j + \hat{y}^{\alpha, -}_{p})\right] \nonumber \\ 
&= \sum^{M}_{p=1} \Bigg\{ \sum_{i \in \mathcal{N}(x_j + \hat{y}^{\alpha, +}_{p})} \frac{1}{2 \Delta x} \left[ A^{\alpha, n}_p w_{i} (x_j + \hat{y}^{\alpha, +}_{p} )\right] (\phi(t_{n}, x_i) - \phi(t_{n}, x_j)) \, +  \nonumber \\ 
&\quad \qquad \sum_{i \in \mathcal{N}(x_j + \hat{y}^{\alpha, -}_{p})} \frac{1}{2 \Delta x} \left[ B^{\alpha, n}_p w_{ i} (x_j + \hat{y}^{\alpha, -}_{p} )\right] (\phi(t_{n}, x_i) - \phi(t_{n}, x_j)) \Bigg\} \nonumber \\ 
& = \sum_{i=1}^N  \sum^{M}_{p=1} \frac{A^{\alpha, n}_p w_{i} (x_j + \hat{y}^{\alpha, +}_{p}) + B^{\alpha, n}_p w_{i} (x_j + \hat{y}^{\alpha, -}_{p})}{2 \Delta x} (\phi(t_{n}, x_i) - \phi(t_{n}, x_j))  \nonumber \\
\label{Lhat}
& = \sum_{i = 1}^N \hat{l}^{\alpha, n}_{j, i}  (\phi(t_{n}, x_i) - \phi(t_{n}, x_j)),
\end{align}
where 
$\mathcal{N}$ is the set of neighbours as in \eqref{eq:monotone_int}, and
\begin{align*}
\hat{l}^{\alpha, n}_{ j, i} = \sum^{M}_{p=1} \frac{A^{\alpha, n}_p w_{ i} (x_j + \hat{y}^{\alpha, +}_{p}(t_{n}, x_j)) + B^{\alpha, n}_p w_{i} (x_j + \hat{y}^{\alpha, -}_{p} (t_{n}, x_j))}{2 \Delta x}.
\end{align*}
The first equality follows from (\ref{eq:TruncScheme}), the second from (\ref{eq:monotone_int}) and since
for all $1\le i, j \le N$ 
\begin{align} \label{eq:monotone_int_2}
w_j(x) \geq 0, \quad w_i(x_j) = \delta_{ij}, \quad \text{and} \quad \sum_{i \in \mathcal{N}(x)} w_i (x) \equiv 1,
\end{align}
for multi-linear interpolation.
Here,
\[
\sum_{i=1}^N \hat{l}^{\alpha, n}_{j, i} = \sum^{M}_{p=1} \frac{A^{\alpha, n}_p + B^{\alpha, n}_p}{2\Delta x} \geq \frac{M}{\Delta x},
\]
with equality only in the absence of domain overstepping for all $p \in [[1, M]]$ at $(t_n, x_j, \alpha)$.

Writing the overall scheme in the form \eqref{eq:positiveScheme} of Definition \ref{def:positiveType}, we have that
\begin{align} \label{eq:FullScheme}
\sup_{\alpha} & \left\{ \left[ 1 + \theta  \Delta t_n \left( \sum^{M}_{p= 1} \frac{A^{\alpha, n}_p + B^{\alpha, n}_p}{2 \Delta x} - \hat{l}^{\alpha, n}_{ j, j} - c^{\alpha, n-1+\theta}_j \right) \right] U^n_j  - \theta \Delta t_n \sum_{i \neq j} \hat{l}^{\alpha, n}_{ j, i} U^n_i + \right. \nonumber \\
& \left. - \left[1 - (1 - \theta) \Delta t_n \left( \sum^{M}_{p= 1} \frac{A^{\alpha, n-1}_p + B^{\alpha, n-1}_p}{2 \Delta x} - \hat{l}^{\alpha, n-1}_{j, j} - c^{\alpha, n-1+\theta}_j \right) \right] U^{n-1}_j + \right.  \nonumber  \\
& \left. - (1-\theta) \Delta t_n \sum_{i \neq j} \hat{l}^{\alpha, n-1}_{j, i} U^{n-1}_i - \Delta t_n f^{\alpha, n-1+\theta}_{j} \right\} =0.
\end{align}
It is straightforward to write down the expressions for the coefficients in \eqref{eq:positiveScheme}:
\begin{align*}
&\mathcal{B}^{\alpha, n, n}_{j, j} = 1 + \theta  \Delta t_n \left( \sum^{M}_{p= 1} \frac{A^{\alpha, n}_p + B^{\alpha, n}_p}{2 \Delta x} - \hat{l}^{\alpha, n}_{ j, j} - c^{\alpha, n-1+\theta}_j \right), \\
&\mathcal{B}^{\alpha, n, n-1}_{ j, j} = 1 - (1 - \theta) \Delta t_n \left( \sum^{M}_{p= 1} \frac{A^{\alpha, n-1}_p + B^{\alpha, n-1}_p}{2 \Delta x} - \hat{l}^{\alpha, n-1}_{j, j} - c^{\alpha, n-1+\theta}_j \right), \\
&\mathcal{B}^{\alpha, n, n}_{j, i} = \theta \Delta t_n \, \hat{l}^{\alpha, n}_{j, i}, \qquad \mathcal{B}^{\alpha, n, n-1}_{j, i} = (1-\theta) \Delta t_n  \, \hat{l}^{\alpha, n-1}_{j, i}.
\end{align*}

\begin{remark}
\label{rem:rhs}
In writing down (\ref{Lhat}), we assumed that the value at the boundary is interpolated from other mesh points, which is feasible on rectangular cuboids, but not for general domain boundaries. In both cases, the Dirichlet boundary value at $x_j + \hat{y}^{\alpha, \pm}_{p}$ can be used.
This has the advantage that interpolation error is avoided. Moreover, as this value then contributes to the right-hand-side $f$ of equation (\ref{eq:FullScheme}) instead of the off-diagonal matrix elements, the system matrix becomes more diagonally dominant. This is advantageous for the iterative solution, see Section \ref{subsec:agmg}.
\end{remark}

The next proposition contains the positivity conditions for the coefficients $\mathcal{B}$ defined above.
\begin{proposition} \label{prop:CFL}
The scheme \eqref{eq:FullScheme} is of positive type if the following conditions hold,
\begin{align} \label{eq:CFL_trunc}
(1 - \theta) \Delta t_n \left[ \sum^{M}_{p=1} \frac{A^{\alpha, n-1}_p + B^{\alpha, n-1}_p}{2\Delta x} - c^{\alpha, n-1+\theta}_i \right] \leq 1, \quad \text{and } \;\; \theta \Delta t_n c^{\alpha, n-1+\theta}_i \leq 1,
\end{align}
for all $\alpha, n, i$.
\end{proposition}

\begin{corollary}
\label{cor:CFL}
In the case of overstepping and $\theta < 1$, monotonicity requires that $\Delta t \sim \mathcal{O}(\Delta x^{3/2})$ if only one side of the diffusion stencils oversteps, or $\Delta t \sim \mathcal{O}(\Delta x^{2})$ if both sides overstep. 
However, if the stencil is not truncated, the positivity condition remains as in \cite{debrabant2013semi}, that is $\Delta t \sim \mathcal{O}(\Delta x)$.
\end{corollary}
\begin{proof}
From Corollary \ref{cor:Growth}, if the corresponding stencil is truncated on one side $A^{\alpha, n-1}_\cdot + B^{\alpha, n-1}_\cdot \sim \mathcal{O}(\Delta x^{-1/2})$ for sufficiently small $\Delta x$, 
$A^{\alpha, n-1}_\cdot + B^{\alpha, n-1}_\cdot \sim \mathcal{O}(\Delta x^{-1})$ if both sides are truncated,
whereas if there is no overstepping, $A^{\alpha, n-1}_\cdot + B^{\alpha, n-1}_\cdot \sim \mathcal{O}(1)$. 
\end{proof}

The $L^\infty$-stability follows from the proof of Lemma 4.1 in \cite{debrabant2013semi} and the new CFL conditions in Proposition \ref{prop:CFL}.

\subsection{Numerical experiments}


To test the truncation of the stencil,
we consider Problems A and B in Section 9.3 from \cite{debrabant2013semi}. 
Both problems follow the formulation in \eqref{eq:1.1}--\eqref{eq:1.3} with homogeneous Dirichlet boundary conditions and have smooth solutions.

\paragraph{Problem A}\hspace{-0.3 cm}(see  Section 9.3 from \cite{debrabant2013semi}). It has exact solution $u(t, x_1, x_2) = \left( \frac{3}{2} - t \right) \sin x_1 \sin x_2,$
and coefficients and control set are given by
\begin{align*}
& f^\alpha = \left( \frac{1}{2} - t \right) \sin x_1 \sin x_2 + \left( \frac{3}{2} - t \right) \Big[ \sqrt{\cos^2 x_1 \sin^2 x_2 + \sin^2 x_1 \cos^2 x_2} +  \\
& \qquad ~  - 2 \sin(x_1 + x_2) \cos (x_1 + x_2) \cos x_1 \cos x_2 \Big], \\
& c^\alpha = 0, ~ b^\alpha = \alpha, \quad \sigma^\alpha = \sqrt{2} \begin{pmatrix} \sin(x_1 + x_2) \\ \cos(x_1 + x_2) \end{pmatrix}, \quad \mathcal{A} = \{ \alpha \in \mathbb{R}^2 : \alpha^2_1 + \alpha^2_2 = 1\}.
\end{align*}

\paragraph{Problem B}\hspace{-0.3 cm}(see  Section 9.3 from \cite{debrabant2013semi}).
\label{pr:ProblemB}
It has exact solution $u(t, x_1, x_2) = \left( 2 - t\right) \sin(x_1) \sin(x_2)$, and coefficients and control set
\begin{align*}
& f^\alpha = \left( 1 - t \right) \sin x_1 \sin x_2 - 2 \alpha_1 \alpha_2 (2 - t) \cos x_1 \cos x_2, \\
& c^\alpha = 0, ~ b^\alpha = 0, \quad \sigma^\alpha = \sqrt{2} \begin{pmatrix} \alpha_1 \\ \alpha_2 \end{pmatrix}, \quad \mathcal{A} = \{ \alpha \in \mathbb{R}^2 : \alpha^2_1 + \alpha^2_2 = 1\}.
\end{align*}

Both problems are solved on the domain $(t, x_1, x_2) \in [0, T] \times [-\pi, \pi]^2$ with $T = \frac{1}{2}$. 
We discretize the spatial domain using Cartesian grids with $N_x \times N_x$ equispaced nodes and for the control set $\mathcal{A}$ we take $N_\alpha$ equally spaced points. Here, $\mathcal{I}_{\Delta x}$ is the usual bilinear interpolator on rectangles. 

For illustration of the stencil and its non-locality, the top row of Figure \ref{fig:stencil_AB} represents the stencil for Problems A and B on a Cartesian grid of $11 \times 11$ points and $10$ points in the control set $\mathcal{A}$. Colour coded lines link the stencil points with the node where the numerical solution is computed, the different colours correspond to the different $\hat{y}^{\alpha, \cdot}_\cdot$. 
On top of some of the stencil points we print the value of the finite difference weights, for compactness we set $A \equiv A^\alpha_{2, 1}(x)$, $B \equiv B^\alpha_{2, 1}(x)$ and $C \equiv (\mu^\alpha_{2, 2}(x))^{-1}$, following the notation in \eqref{eqn_Ap} and \eqref{eqn_APp1}.
The bottom row of Figure \ref{fig:stencil_AB} represents the non-locality of the diffusion stencil by counting the number of stencil points at a given distance from the central node. The distance is measured as multiples of $\Delta x$ and given by $\left\lfloor \frac{(\sigma^\alpha(x))_i}{\sqrt{\Delta x}} \right\rfloor$, where the grid is of size $641 \times 641$ and $10$ points in the control set $\mathcal{A}$.

\begin{figure}[htp]
	\centering
	\subfloat[Stencil for Problem A in \cite{debrabant2013semi} on a Cartesian $11 \times 11$ grid for 10 sample points in the control set $\mathcal{A}$.]
	{\label{subfig:stencil_3A}\includegraphics[width=0.49\textwidth]{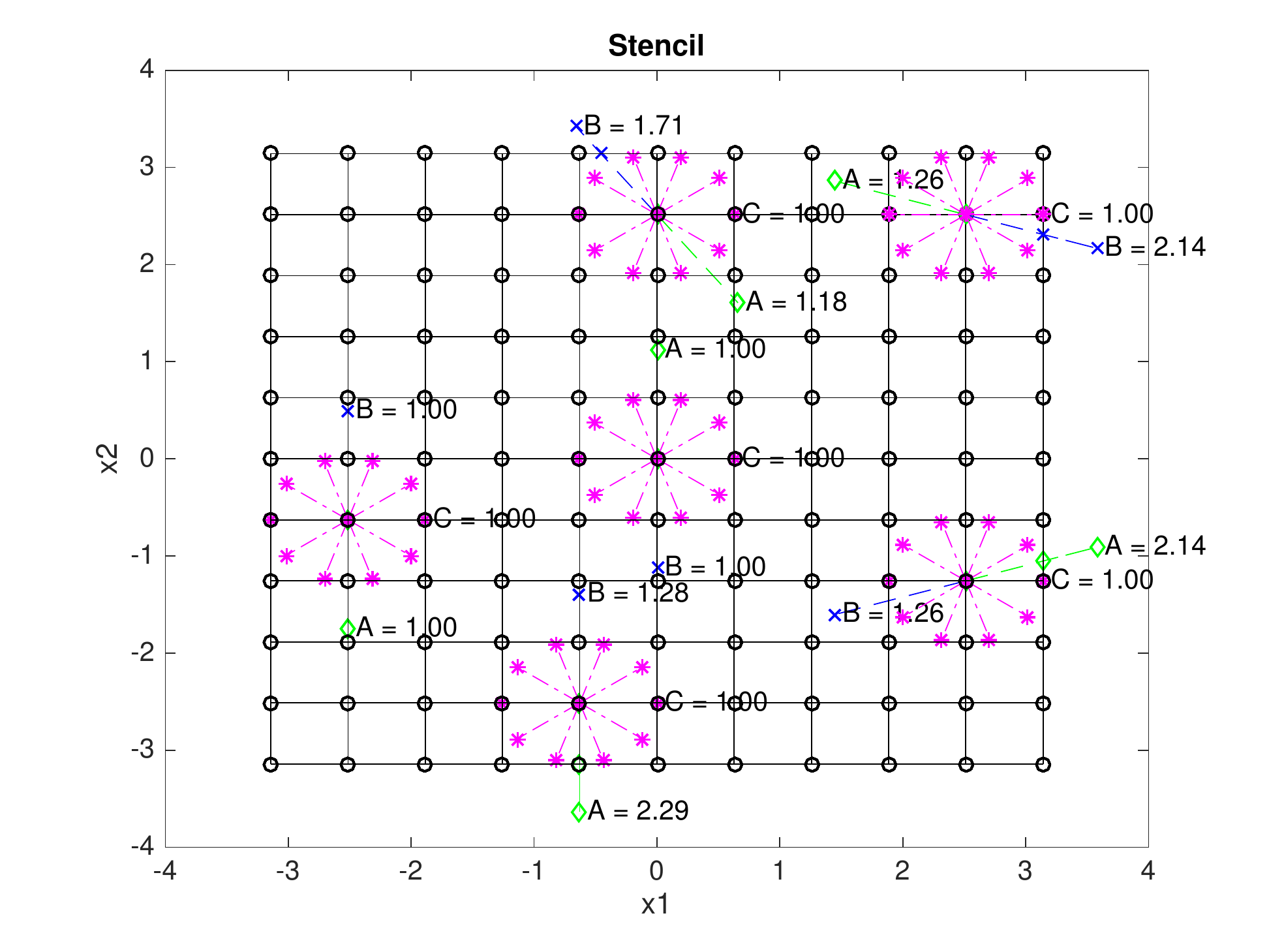}}
	\hspace{\stretch{1}}
	\subfloat[Stencil for Problem B in \cite{debrabant2013semi} on a Cartesian $11 \times 11$ grid for 10 sample points in the control set $\mathcal{A}$.]
	{\label{subfig:stencil_3B}\includegraphics[width=0.49\textwidth]{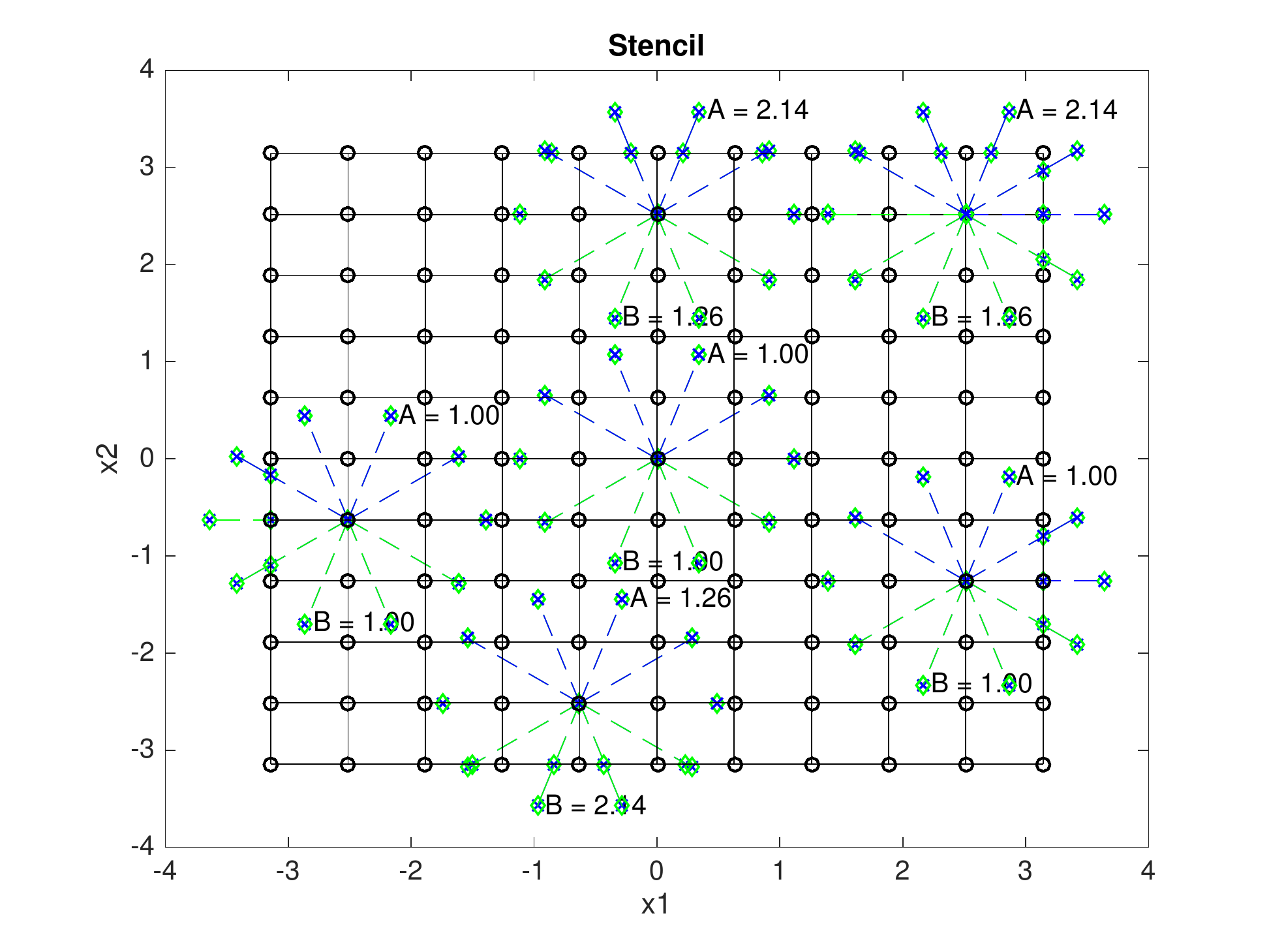}}
	\\
	\vspace{2em}
	\subfloat[Histogram of $\left\lfloor \frac{(\sigma^\alpha(x))_i}{\sqrt{\Delta x}} \right\rfloor$ in Problem A for all $x \in \Omega_{\Delta x}$ where $\Omega_{\Delta x}$ is a Cartesian grid with $\Delta x = \frac{2\pi}{640}$, 10 points in the control set $\mathcal{A}$, and $i \in \{1, 2\}$ is the dimension index.]
	{\label{subfig:stencilHist_3A}\includegraphics[width=0.49\textwidth]{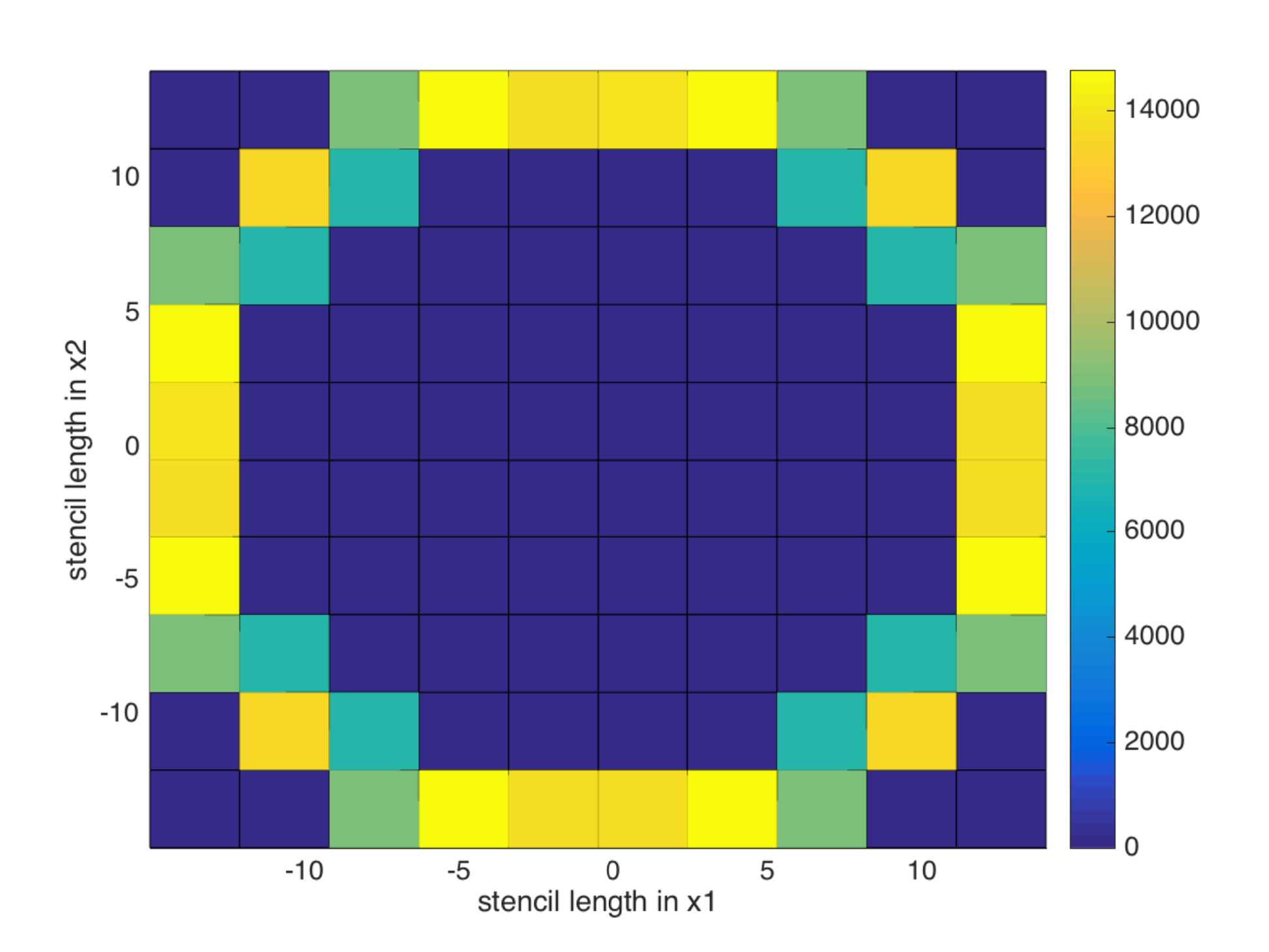}}
	\hspace{\stretch{1}}	
		\subfloat[Histogram of $\left\lfloor \frac{(\sigma^\alpha(x))_i}{\sqrt{\Delta x}} \right\rfloor$ in Problem B for all $x \in \Omega_{\Delta x}$ where $\Omega_{\Delta x}$ is a Cartesian grid with $\Delta x = \frac{2\pi}{640}$, 10 points in the control set $\mathcal{A}$, and $i \in \{1, 2\}$ is the dimension index.]
	{\label{subfig:stencilHist_3B}\includegraphics[width=0.49\textwidth]{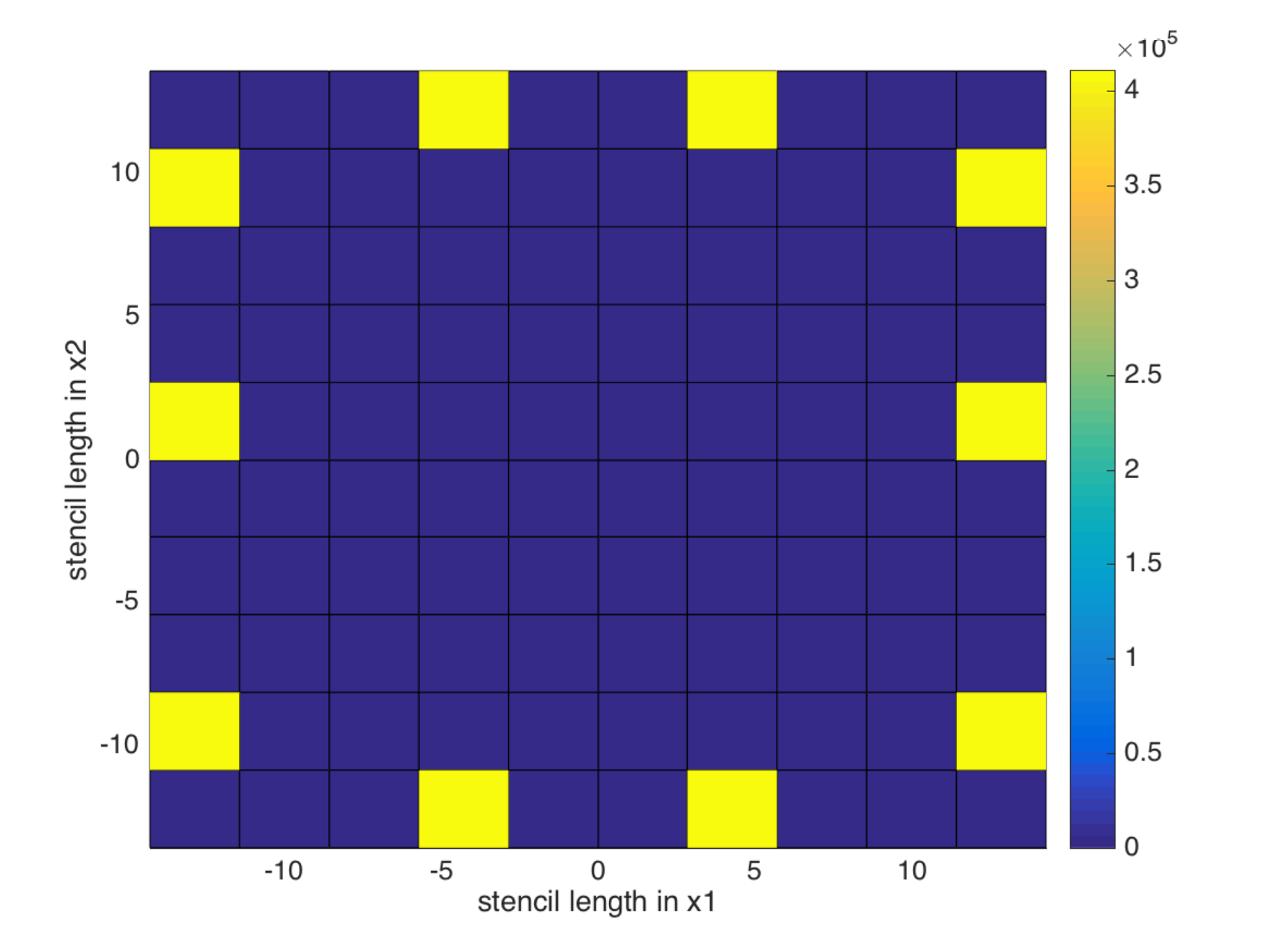}}
		\caption{Graphical representation of the stencil over a two-dimensional Cartesian grid of size $11 \times 11$ and 10 equally spaced points in the control set $\mathcal{A}$. The finite difference weights corresponding to some of the points are printed, where for simplicity the weights are labelled $A \equiv A^\alpha_{2, 1}(x)$, $B \equiv B^\alpha_{2, 1}(x)$ and $C \equiv (\mu^\alpha_{2, 2}(x))^{-1}$, following the notation in \eqref{eqn_Ap} and \eqref{eqn_APp1}. To illustrate the non-locality of the scheme as the grid is refined, the second row represents the histograms of the shortest displacement from the central node for a grid of size $641 \times 641$ for both problems. The radius of the stencil in $\sigma^\alpha$ is 14.27 for this grid, given by $\frac{\|\sigma^\alpha\|_2}{\sqrt{\Delta x}} = \sqrt{640/\pi}$.}
	\label{fig:stencil_AB}
\end{figure}

Problems A and B were obviously chosen in \cite{debrabant2013semi} for their periodic solutions, to be able to analyse the convergence of the scheme without the complication of boundary conditions. Here, we do not make use of the periodicity but only use the values at the boundary and not outside the domain.

We note that the problems being linear in $t$, a single time step with $\Delta t = T$ suffices to obtain an exact solution in $t$. However, in order to check the effect of the truncation on the stability, in addition to $\Delta t = T$, we also investigate $\Delta t$ equal to $\frac{\Delta x}{4}$, $\Delta x^{3/2}$, and $\Delta x^2$. We report the $\infty$-norm of the errors over two regions: the first one comprising the whole domain, and the second one comprising part of the interior of the domain.

We consider explicit and implicit time stepping schemes, corresponding to $\theta = 0$ and $\theta = 1$ respectively. For the explicit scheme in the case of overstepping we test the following modifications of the scheme:

\begin{enumerate}
\item truncation of the stencil as discussed in Section \ref{sec:Consistency} (Table \ref{Tab:Linf_matrix_40} for Problem A and Table \ref{Tab:Linf_exp_trunc_40_B} for Problem B);
\item constant extrapolation of the boundary value 
 in the direction of the semi-Lagrangian step
(Table \ref{Tab:Linf_const_40} for Problem A and Table \ref{Tab:Linf_exp_const_40_B} for Problem B);
\item linear extrapolation of the boundary value in the direction of the semi-Lagrangian step (Table \ref{Tab:Linf_linear_40} for Problem A and Table \ref{Tab:Linf_exp_lin_40_B} for Problem B).
\end{enumerate}
For the implicit case we only consider the first modification, i.e.\ truncation of the stencil (Table \ref{Tab:Linf_imp_trunc_40} for Problem A and Table \ref{Tab:Linf_imp_trunc_40_B} for Problem B). 

The results confirm the impact of the truncation on the stability of the scheme, when $\theta = 0$. 
However, when $\theta = 1$, we do not observe any instability regardless of the size of the time step. When stable, the truncation of the stencil outperforms the two extrapolations of the boundary conditions considered. 
Furthermore, as the mesh and time steps are refined, only the truncated scheme, if stable, achieves convergence orders close to $\mathcal{O}(\Delta x)$ when the error at $t = T$ is measured on the entire spatial grid.
This can be explained without rigorous proof by the observation that  the truncation error of order $\sqrt{\Delta x}$ is restricted to a boundary layer of width $\sqrt{\Delta x}$.
Therefore, as seen from the last two columns in Table \ref{Tab:Linf_imp_trunc_40}, choosing $\Delta t$ of order higher than 1 in $\Delta x$ does not improve the accuracy of the numerical results and leads to computational inefficiency.

\begin{remark}
Regarding the discretization of the control set, we take $N_\alpha = 40$ equally spaced points. For this choice, the discretization error of the LISL scheme is found to dominate the control discretization error for the problems and the space-time mesh sizes considered.
\end{remark}



\def\ignore#1\endignore{}
\newcolumntype{I}{@{}>{\ignore}c<{\endignore}} 

\newcommand{\newsubtab}[3]{
\subfloat[#2]{
\begin{tabular}{c|cr|cr|cr|cr}
\multirow{2}{*}{$N_x$}&\multicolumn{2}{c|}{$\Delta t = T$}&\multicolumn{2}{c|}{$\Delta t = \frac{\Delta x}{4}$}&\multicolumn{2}{c|}{$\Delta t = \Delta x^{\frac{3}{2}}$}&\multicolumn{2}{c}{$\Delta t = \Delta x^2$}\\ 
\cline{2-9}
&error&rate&error&rate&error&rate&error&rate\\\hline
#1
\end{tabular}
}}

\begin{table}[htp]\captionsetup[subfloat]{position=top}
\begin{center}
\newsubtab{
41 & 1.42e-01 & - & 4.39e-02 & - & 4.39e-02 & - & 4.36e-02 & -\\
81 & 1.04e-01 & 0.45 & 2.12e-02 & 1.05 & 2.11e-02 & 1.06 & 2.11e-02 & 1.05\\
161 & 7.36e-02 & 0.50 & 1.10e-02 & 0.94 & 1.10e-02 & 0.94 & 1.10e-02 & 0.94\\
321 & 5.28e-02 & 0.48 & 1.34e+23 & -83.33 & 5.77e-03 & 0.93 & 5.76e-03 & 0.93\\
641 & 3.77e-02 & 0.48 & 5.07e+89 & -221.17 & 3.10e-03 & 0.90 & 3.10e-03 & 0.89\\
}
{Error in $L^\infty$-norm over $\Omega_{\Delta x}$}{}
\\\newsubtab{
41 & 8.61e-02 & - & 4.38e-02 & - & 4.42e-02 & - & 4.35e-02 & -\\
81 & 4.22e-02 & 1.03 & 2.12e-02 & 1.05 & 2.11e-02 & 1.06 & 2.11e-02 & 1.05\\
161 & 2.14e-02 & 0.98 & 1.10e-02 & 0.94 & 1.10e-02 & 0.95 & 1.10e-02 & 0.94\\
321 & 1.10e-02 & 0.96 & 1.84e+13 & -50.57 & 5.71e-03 & 0.95 & 5.70e-03 & 0.95\\
641 & 5.96e-03 & 0.88 & 1.06e+72 & -195.20 & 3.08e-03 & 0.89 & 3.08e-03 & 0.89\\
}
{Error in $L^\infty$-norm over $\Omega_{\Delta x} \cap \lbrack -\pi/2 , \pi/2 \rbrack^2$}{}
\end{center}
\caption{Results using the truncation of the stencil for explicit method with $N_{\alpha} = 40$ for Problem A.\label{Tab:Linf_matrix_40}}
\end{table}

\begin{table}[htp]\captionsetup[subfloat]{position=top}
\begin{center}
\newsubtab{
41 & 1.36e+00 & - & 3.68e-01 & - 		&   3.72e-01 & - & 3.65e-01 & - \\
81 & 1.89e+00 & -0.48 & 2.61e-01 & 0.49 &   2.62e-01 & 0.51 & 2.60e-01 & 0.49\\
161 & 2.67e+00 & -0.49 & 1.80e-01 & 0.54 &  1.80e-01 & 0.54 & 1.80e-01 & 0.53\\
321 & 3.77e+00 & -0.50 & 1.27e-01 & 0.51 &  1.27e-01 & 0.51 & 1.27e-01 & 0.51\\
641 & 5.34e+00 & -0.50 & 9.18e-02 & 0.47 &  9.18e-02 & 0.47 & 9.18e-02 & 0.46\\
}
{Error in $L^\infty$-norm over $\Omega_{\Delta x}$}{Tab:Ex1a}
\\\newsubtab{
41 & 1.59e-01 & - & 1.04e-01 & - & 		 1.05e-01 & - & 1.03e-01 & -\\
81 & 8.15e-02 & 0.96 & 5.25e-02 & 0.99 & 5.26e-02 & 1.00 & 5.22e-02 & 0.98\\
161 & 4.22e-02 & 0.95 & 2.67e-02 & 0.98 & 2.66e-02 & 0.98 & 2.66e-02 & 0.97\\
321 & 2.18e-02 & 0.95 & 1.36e-02 & 0.97 & 1.36e-02 & 0.97 & 1.36e-02 & 0.97\\
641 & 1.21e-02 & 0.85 & 8.21e-03 & 0.73 & 8.20e-03 & 0.73 & 8.19e-03 & 0.73\\

}
{Error in $L^\infty$-norm over $\Omega_{\Delta x} \cap \lbrack -\pi/2 , \pi/2 \rbrack^2$}{Tab:Ex1b}
\end{center}
\caption{Results using constant extrapolation of the boundary condition for explicit method with $N_{\alpha} = 40$ for Problem A.\label{Tab:Linf_const_40}}
\end{table}

\begin{table}[htp]\captionsetup[subfloat]{position=top}
\begin{center}
\newsubtab{
41 & 1.59e-01 & - & 1.04e-01 & - 		&  1.05e-01 & -  	  & 1.03e-01 & -\\
81 & 8.15e-02 & 0.96 & 5.25e-02 & 0.99 	&  5.26e-02 & 1.00     & 5.22e-02 & 0.98\\
161 & 4.28e-02 & 0.93 & 5.62e-01 & -3.42 &  5.63e-01 & -3.42  & 5.58e-01 & -3.42\\
321 & 2.75e-02 & 0.64 & 4.41e+03 & -12.94 & 6.00e+03 & -13.38 & 8.00e+03 & -13.81\\
641 & 1.85e-02 & 0.57 & 2.77e+20 & -55.80 & 2.70e+20 & -55.32 & 1.37e+21 & -57.25\\
}
{Error in $L^\infty$-norm over $\Omega_{\Delta x}$}{Tab:Ex1a}
\\\newsubtab{
41 & 1.59e-01 & - & 1.04e-01 & - 			& 1.05e-01 & -   & 1.03e-01 & -\\
81 & 8.15e-02 & 0.96 & 5.25e-02 & 0.99 		& 5.26e-02 & 1.00    & 5.22e-02 & 0.98\\
161 & 4.22e-02 & 0.95 & 2.67e-02 & 0.98 	&  2.66e-02 & 0.98   & 2.66e-02 & 0.97\\
321 & 2.18e-02 & 0.95 & 1.96e+00 & -6.20 	&  2.07e+00 & -6.28   & 2.23e+00 & -6.39\\
641 & 1.21e-02 & 0.85 & 9.26e+14 & -48.75 	&  3.18e+15 & -50.45  & 3.01e+15 & -50.26\\
}
{Error in $L^\infty$-norm over $\Omega_{\Delta x} \cap \lbrack -\pi/2 , \pi/2 \rbrack^2$}{Tab:Ex1b}
\end{center}
\caption{Results using linear extrapolation for points out of the domain for explicit method with $N_{\alpha} = 40$ for Problem A.\label{Tab:Linf_linear_40}}
\end{table}

\begin{table}[htp]\captionsetup[subfloat]{position=top}
\begin{center}
\newsubtab{
41 & 3.25e-02 & - & 4.21e-02 & -         & 4.17e-02 & - & 4.24e-02 & -\\
81 & 1.59e-02 & 1.03 & 2.08e-02 & 1.02   & 2.08e-02 & 1.01 & 2.09e-02 & 1.02\\
161 & 8.39e-03 & 0.92 & 1.09e-02 & 0.93  & 1.09e-02 & 0.93 & 1.10e-02 & 0.93\\
321 & 4.38e-03 & 0.94 & 5.75e-03 & 0.93  & 5.75e-03 & 0.93 & 5.76e-03 & 0.93\\
641 & 2.37e-03 & 0.89 & 3.09e-03 & 0.89  & 3.10e-03 & 0.89 & 3.10e-03 & 0.89\\
}
{Error in $L^\infty$-norm over $\Omega_{\Delta x}$}{Tab:Ex1a}
\\\newsubtab{
41 & 3.25e-02 & - & 4.21e-02 & - &        4.17e-02 & -  & 4.24e-02 & -\\
81 & 1.59e-02 & 1.03 & 2.08e-02 & 1.02 &  2.08e-02 & 1.01  & 2.09e-02 & 1.02\\
161 & 8.39e-03 & 0.92 & 1.09e-02 & 0.93 & 1.09e-02 & 0.93 & 1.10e-02 & 0.93\\
321 & 4.35e-03 & 0.95 & 5.68e-03 & 0.94 & 5.69e-03 & 0.94 & 5.70e-03 & 0.95\\
641 & 2.37e-03 & 0.88 & 3.07e-03 & 0.89 & 3.08e-03 & 0.89 & 3.08e-03 & 0.89\\
}
{Error in $L^\infty$-norm over $\Omega_{\Delta x} \cap \lbrack -\pi/2 , \pi/2 \rbrack^2$}{Tab:Ex1b}
\end{center}
\caption{Results using truncation for points out of the domain for implicit method with $N_{\alpha} = 40$ for Problem A.\label{Tab:Linf_imp_trunc_40}}
\end{table}

\begin{remark}
\label{rem:two-over}
Corollary \ref{cor:CFL} shows two different CFL conditions for the truncated stencil, the first one for diffusion stencils where only one side oversteps and a second one when both sides overstep. 
The results in Table \ref{Tab:Linf_matrix_40} for Problem A and Table \ref{Tab:Linf_exp_trunc_40_B} for Problem B correspond to the former situation. 
To check the sharpness of the latter, we shift the spatial domain in Problem A in both directions by $\frac{7\pi}{8}$.
The new spatial domain is thus $\bar{\Omega} = [\frac{-\pi}{8}, \frac{15\pi}{8}]^2$. 
Note that the solution itself is periodic with period $2\pi$.
This problem differs from the original one in that both sides of the diffusion stencil overstep for mesh points
within a distance of $\mathcal{O}(\sqrt{\Delta x})$ to
the bottom left corner, located at $(\frac{-\pi}{8}, \frac{-\pi}{8})$,
where $\sigma^\alpha = (-1,1)^T$.
In Table \ref{Tab:shifted} we report the results for the explicit method using the truncation of the stencil. 
As expected, we find that we now need $\Delta t \sim \Delta x^2$ for stability.
\end{remark}

\begin{table}[htp]\captionsetup[subfloat]{position=top}
\begin{center}
\newsubtab{
41 & 1.55e-01 & - & 4.71e-02 & - & 4.76e-02 & - & 4.67e-02 & -\\
81 & 1.12e-01 & 0.47 & 1.57e+05 & -21.67 & 7.90e+05 & -23.98 & 2.11e-02 & 1.15\\
161 & 8.04e-02 & 0.47 & 1.02e+33 & -92.39 & 1.30e+35 & -97.06 & 1.10e-02 & 0.94\\
321 & 5.80e-02 & 0.47 & 6.73e+103 & -235.26 & 5.96e+138 & -344.35 & 5.76e-03 & 0.93\\
641 & 4.22e-02 & 0.46 & 8.17e+276 & -574.97 & NaN & NaN & 3.10e-03 & 0.89\\
}
{Error in $L^\infty$-norm over $\Omega_{\Delta x}$}{}
\\\newsubtab{
41 & 8.65e-02 & - & 4.70e-02 & - & 4.74e-02 & - & 4.66e-02 & -\\
81 & 4.22e-02 & 1.04 & 2.07e-02 & 1.18 & 2.07e-02 & 1.19 & 2.06e-02 & 1.18\\
161 & 2.14e-02 & 0.98 & 1.18e+06 & -25.76 & 1.18e+09 & -35.73 & 1.08e-02 & 0.93\\
321 & 1.10e-02 & 0.96 & 7.99e+47 & -138.96 & 4.94e+84 & -251.21 & 5.59e-03 & 0.95\\
641 & 5.96e-03 & 0.88 & 9.81e+165 & -392.28 & NaN & NaN & 3.02e-03 & 0.89\\
}
{Error in $L^\infty$-norm over $\Omega_{\Delta x} \cap \lbrack 3\pi/8 , 11\pi/8 \rbrack^2$}{}
\end{center}
\caption{Results using the truncation of the stencil for explicit method with $N_{\alpha} = 40$ for Problem A on a shifted domain,
as described in Remark \ref{rem:two-over}.
\label{Tab:shifted}}
\end{table}

\begin{remark}
For the explicit method using the truncation of the stencil, i.e.\ Tables \ref{Tab:Linf_matrix_40}, \ref{Tab:Linf_exp_trunc_40_B} and \ref{Tab:shifted}, focusing on the $\Delta t = T$ case,  we notice that the convergence rate over the whole mesh $\Omega_{\Delta x}$ is approximately 0.5, whereas it is approximately 1.0 when the error is measured in the interior of the mesh.
We also notice that there is a significant difference between the magnitude of the errors if measured over the whole grid or on a region in the interior.
The difference in the magnitude of the errors may be due to the fact that $\Delta t = T$ does not satisfy the CFL condition and that the CFL condition is more restrictive for points where the stencil is truncated. 
It is also at these points that the local consistency error is of order $\sqrt{\Delta x}$ as shown in Corollary \ref{cor:Growth}.
The situation is different for $\Delta t = \Delta x^2$ in the explicit case, or for any $\Delta t$ in the implicit case.
In these cases, the error convergence rates are approximately 1.0 when measured over the whole mesh and the errors over the whole grid and in the interior are comparable in magnitude. 
\end{remark}

\FloatBarrier

\section{Multigrid preconditioning}
\label{sec:MG}

In this section, we study the application of multigrid preconditioners together with policy iteration \cite{bokanowski_Howard} to solve the non-linear system \eqref{eq:nonlinear_sys}.

Geometric multigrid requires us to predefine a grid hierarchy based on the geometry of the problem. The variability of the width of the LISL stencil within a given grid (variable coefficients) and through the grid hierarchy makes it difficult, even for simple problems, to design an appropriate grid hierarchy and a good smoother. Moreover, the varying stencil requires us to build the coarse-grid version of the operator algebraically instead of using its coarse grid version, which further limits our knowledge of the problem as we go deeper into the grid hierarchy.

Another aspect to consider is related to the transfer operators. Standard grid interpolations provide approximations using the grid neighbours of a given node, whereas for the LISL stencil, being non-local, the solution at a given node may not be best approximated by its neighbours on the grid but by those on its stencil. These heuristics suggest that the algebraic approach to multigrid, fixing the smoother and building operator dependent intergrid transfer operators, may result in more efficient multigrid preconditioning for LISL discretizations.

Algebraic multigrid (AMG), introduced in \cite{ruge1987algebraic}, constructs ``coarse grids" based on the matrix coefficients. However, as pointed out in Section 6.2 of the recent review on preconditioning \cite{wathen2015precond}, AMG coarsening may not reduce the number of variables fast enough from one grid to the next. 
A slow reduction in the number of unknowns and the use of the Galerkin principle to build the coarse system matrix with intergrid transfer operators using weighted averages increase the complexity of the multigrid scheme. To measure the complexity the following quantities are commonly used:
\begin{definition} 
\label{gridc}
The grid complexity  $c_G$ is the total number of variables $N_.$, on all multigrid levels, divided by the number of variables on the finest level $N_1$,
\begin{align*}
c_G = \frac{1}{N_1} \sum_{\ell = 1}^{n_{\text{levels}}} N_\ell.
\end{align*}
\end{definition}
\begin{definition} 
\label{algebrac}
The algebraic complexity $c_A$ is the total number of non-zero entries, in all matrices $A_\ell$, divided by the number of non-zero entries of the finest level operator $A_1$,
\begin{align*}
c_A = \frac{1}{\rm{nnz}(A_1)} \sum_{\ell = 1}^{n_{\text{levels}}} \rm{nnz}(A_\ell).
\end{align*}
\end{definition}
%

We will find a benefit to the convergence of constructing the ``coarse grids"  algebraically already for simple  examples of LISL matrices
(Section \ref{sec:NE_GMG}, in particular Table \ref{Tab:rhoMG}),
and that algebraic construction of the grid hierarchy deals well with the varying LISL stencils (Section \ref{subsec:perform}).
However, there is an increase in complexity of AMG (see Table \ref{Tab:rhoMG}) mainly due to the use of interpolation in LISL discretizations.

Recent and on-going research on algebraic multigrid \cite{notay2010aggregation, notay2012aggregation} shows how one can construct good multigrid cycles using simplified ``intergrid" transfer operators based on aggregation of the unknown variables, thus avoiding the problem
of increased complexity on coarser levels. In particular, \cite{notay2012aggregation} proves convergence of a simplified two-grid scheme using aggregation for non-singular M-matrices with non-negative row and column sums. We will show that these results apply for LISL discretizations matrices and justify the use of AGMG both theoretically and empirically.

\subsection{On the spectrum of LISL matrices}

To assess the suitability of preconditioning based on geometric multigrid, we start by considering the spectrum of LISL matrices for a simplified model. 
For illustration, we first calculate the eigenvalues and eigenvectors of the LISL discretization of the diffusion operator with constant coefficients, for any function $u : \mathbb{R}^d \to \mathbb{R}$
\begin{align} \label{eq:laplacian}
- \frac{1}{2} \nabla^T (\sigma \sigma^T) \nabla u = - \frac{1}{2} \sum_{i = 1}^d  \sigma^2_i \frac{\partial^2 u}{\partial x^2_i},
\end{align}
where $\sigma \in \mathbb{R}^{d \times d}$ is a diagonal matrix with $(\sigma)_{ii} = \sigma_i$.


We start by considering the one-dimensional case on an equispaced grid $\Omega_{\Delta x}$, where $\Delta x > 0$ is the distance between two consecutive nodes. For $\sigma > 0$, we define 
\begin{align} \label{eq:def_mAlpha}
m \defeq  \left\lfloor \frac{\sigma}{\sqrt{\Delta x}} \right\rfloor, \quad \text{ and} \quad \gamma \defeq  (m+1) - \frac{\sigma}{\sqrt{\Delta x}},
\end{align}
where $m \in \mathbb{N}$ denotes the stencil length and $\gamma	\in [0, 1]$ is the interpolation weight of the one-dimensional linear interpolation operator, such that for any real function $\phi : \mathbb{R} \to \mathbb{R}$ the linear interpolation operator on $\Omega_{\Delta x}$ is $\mathcal{I}_{\Delta x} (\phi) (x_i + \sqrt{\Delta x} \sigma) = \gamma \phi(x_i + m \Delta x) + (1-\gamma) \phi(x_i + (m+1) \Delta x)$. Without loss of generality and for simplicity of the notation we assume that $\Delta x = 1$. Denote by $L_N$ the following $N \times N$ Laplacian matrix
\begin{equation*}
L_N \defeq \begin{pmatrix}
2 & -1 & 0 & \cdots & 0 \\
-1 & 2 & -1 & \cdots & 0 \\
0 & -1 & 2 & \cdots & 0 \\
\vdots & \vdots & \vdots & \ddots & \vdots \\
0 & 0 & 0 & \cdots & 2
\end{pmatrix}.
\end{equation*}

Let now $m = 2$ and $\Delta x = 1$, then the LISL discretization matrix is given by
\begin{equation}
\label{eq:LSL}
L^{N, m, \gamma}_{SL} \defeq \begin{pmatrix}
2 & 0 & -\gamma & -1 + \gamma & 0 &\cdots & 0 \\
0 & 2 & 0 & -\gamma & -1 + \gamma & \cdots & 0 \\
-\gamma & 0 & 2 & 0 & -\gamma & \cdots & 0 \\
\vdots & \vdots & \vdots & \vdots & \vdots & \ddots & \vdots \\
0 & 0 & 0 & 0 & 0 & \cdots & 2
\end{pmatrix}.
\end{equation}

Noticing the structure in the diagonals, we re-write $L^{N, m, \gamma}_{SL}$
as
\begin{align*}
L^{N, m, \gamma}_{SL} 
=~ & \gamma L^m_N + (1-\gamma) L^{m+1}_N,
\end{align*}
where $L^m_N = L^{N, m, 1}_{SL} $.

Using the properties of Kronecker products we can characterize the eigenvalues of the matrices $L^m_N$ in terms of the eigenvalues of the standard $L_N$ matrices. Denoting by $\lambda(L_N) \in \mathbb{R}^N$ and $V(L_N)\in \mathbb{R}^{N\times N}$ the eigenvalues and eigenvectors of $L_N$, respectively, we have that
\begin{align*}
\lambda(L^m_N) &= 
\left[\lambda \left( L_{\left\lceil \frac{N}{m}  \right\rceil} \right) \otimes e_1 \right]_N
+ 
\left[
\lambda \left( L_{\left\lfloor \frac{N}{m}  \right\rfloor} \right) \otimes \sum_{i=2}^N e_i \right]_N, \\
V(L^m_N)  &= \left[ V \left( L_{\left\lceil \frac{N}{m}  \right\rceil} \right)  \otimes \begin{pmatrix}
1 & 0 \\
0 & \bm{0}_{m-1}
\end{pmatrix}  \right]_{N \times N} +
   \left[
V \left( L_{\left\lfloor \frac{N}{m}  \right\rfloor} \right) \otimes \begin{pmatrix}
0 & 0 \\
0 & I_{m-1}
\end{pmatrix} \right]_{N\times N},
\end{align*}
where $e_i$ is the $i$-th canonical basis vector of $\mathbb{R}^N$,
$I_N$ is the $N \times N$ identity matrix and $\bm{0}_{m}$ denotes the $m \times m$ zero matrix. By $[ A ]_{N \times N}$ we mean that we select the first $N$ rows and $N$ columns of $A$, and similar for $[v]_{N}$ for a vector $N$.
This is required as $N$ will in general not be a multiple of both $m$ and $m+1$ so the resulting matrices from the Kronecker product will be of size $\left\lceil \frac{N}{m}  \right\rceil m$ and $\left\lceil \frac{N}{m+ 1}  \right\rceil (m +1)$ which are greater or equal to $N$. 

In the presence of interpolation, that is, when $\gamma \in (0, 1)$, we are unable to provide any closed formula to the eigenvalues and eigenvectors of $L^{N, m, \gamma}_{SL} = \gamma L^m_N + (1-\gamma) L^{m+1}_N$. Figure \ref{fig:eigvectors} contains graphs with the eigenvalues and some eigenvectors of the matrices $L^{N, m, \gamma}_{SL}$, $L^m_N$, $L^{m+1}_N$ and $L_N$. 
The plots show that for LISL discretization matrices, in contrast to the standard case, small eigenvalues are not necessarily associated with smooth modes.
As a result, these components cannot be represented accurately on the coarse mesh.


\begin{figure}[htp]
	\centering
	\subfloat[Comparison of the eigenvalues of $L^{31, 5, 0.346}_{SL}$, $L^5_{31}$, $L^6_{31}$ and $L_{31}$
	in increasing order.]
	{\label{subfig:eigenvalues}\includegraphics[width=0.49\textwidth]{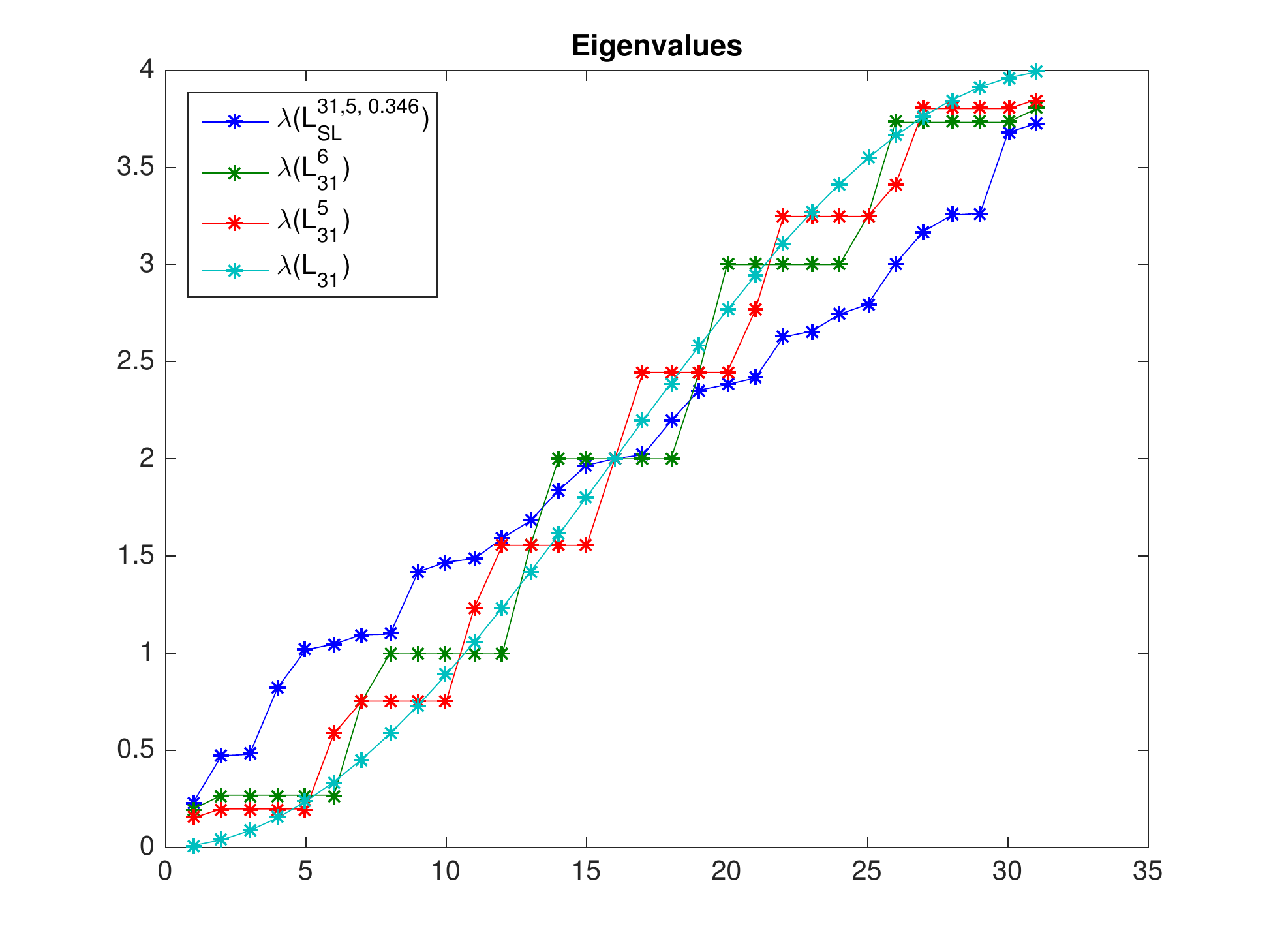}}
	\hspace{\stretch{1}}
	\subfloat[Eigenvectors corresponding to the smallest eigenvalue for $L^{31, 5, 0.346}_{SL}$, $L^5_{31}$, $L^6_{31}$ and $L_{31}$.]
	{\label{subfig:eigvector_1}\includegraphics[width=0.49\textwidth]{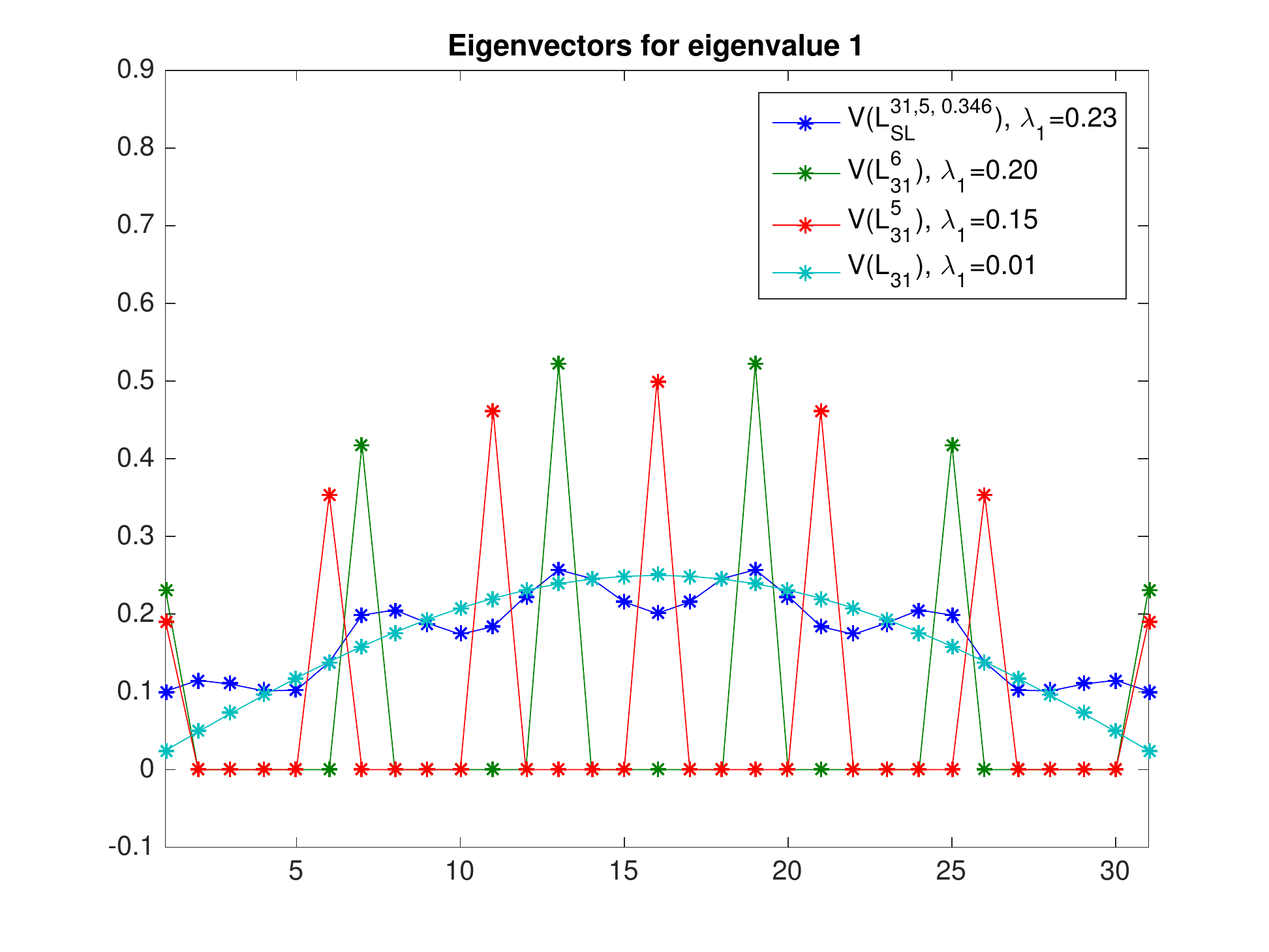}}
	\\
	\vspace{2em}
		\subfloat[Eigenvectors corresponding to the 15-th eigenvalue 
		for $L^{31, 5, 0.346}_{SL}$, $L^5_{31}$, $L^6_{31}$ and $L_{31}$.]
	{\label{subfig:eigvector_15}\includegraphics[width=0.49\textwidth]{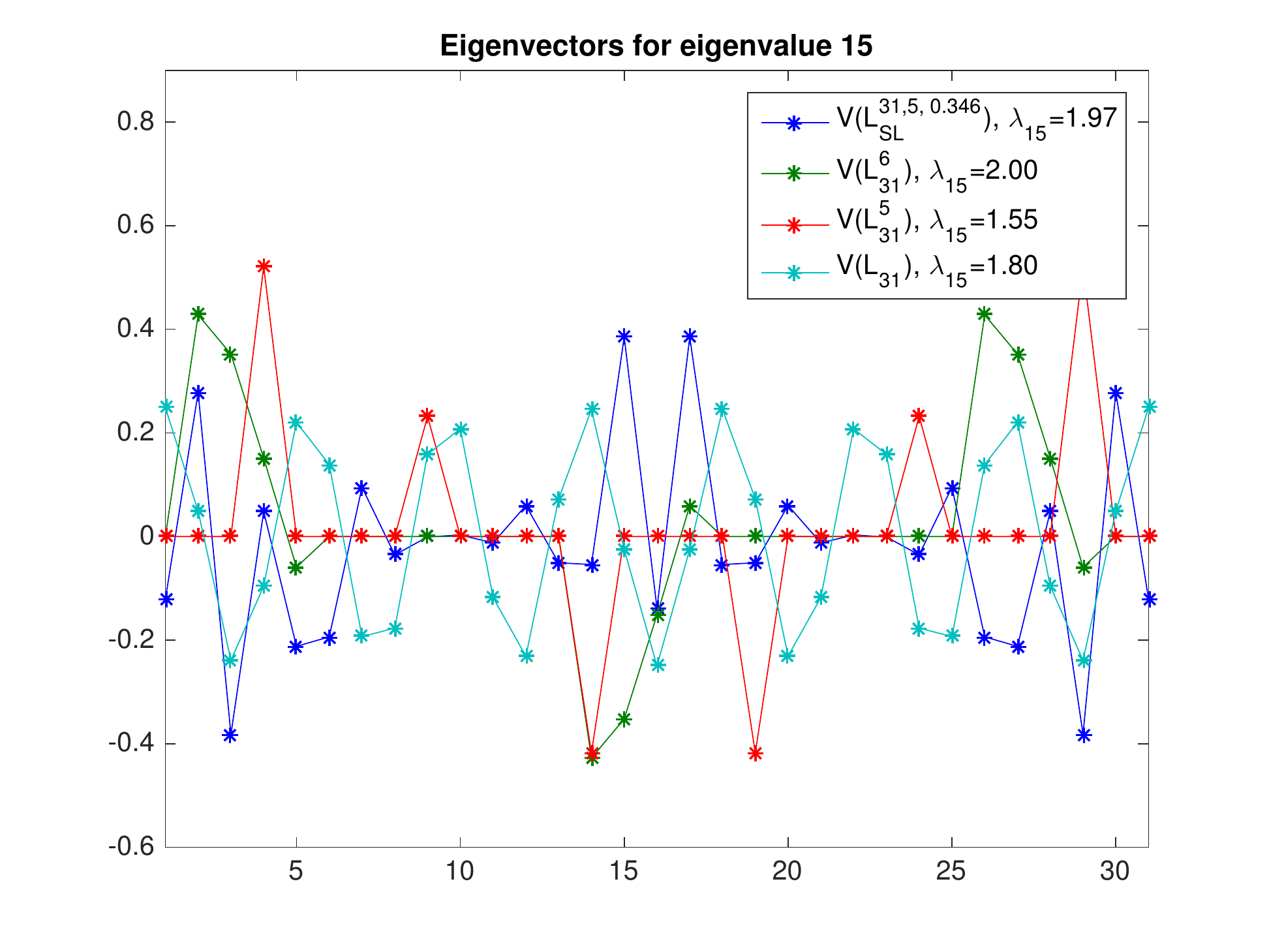}}
	\hspace{\stretch{1}}
	\subfloat[Eigenvectors corresponding to the largest eigenvalue for $L^{31, 5, 0.346}_{SL}$, $L^5_{31}$, $L^6_{31}$ and $L_{31}$.]
	{\label{subfig:eigvector_31}\includegraphics[width=0.49\textwidth]{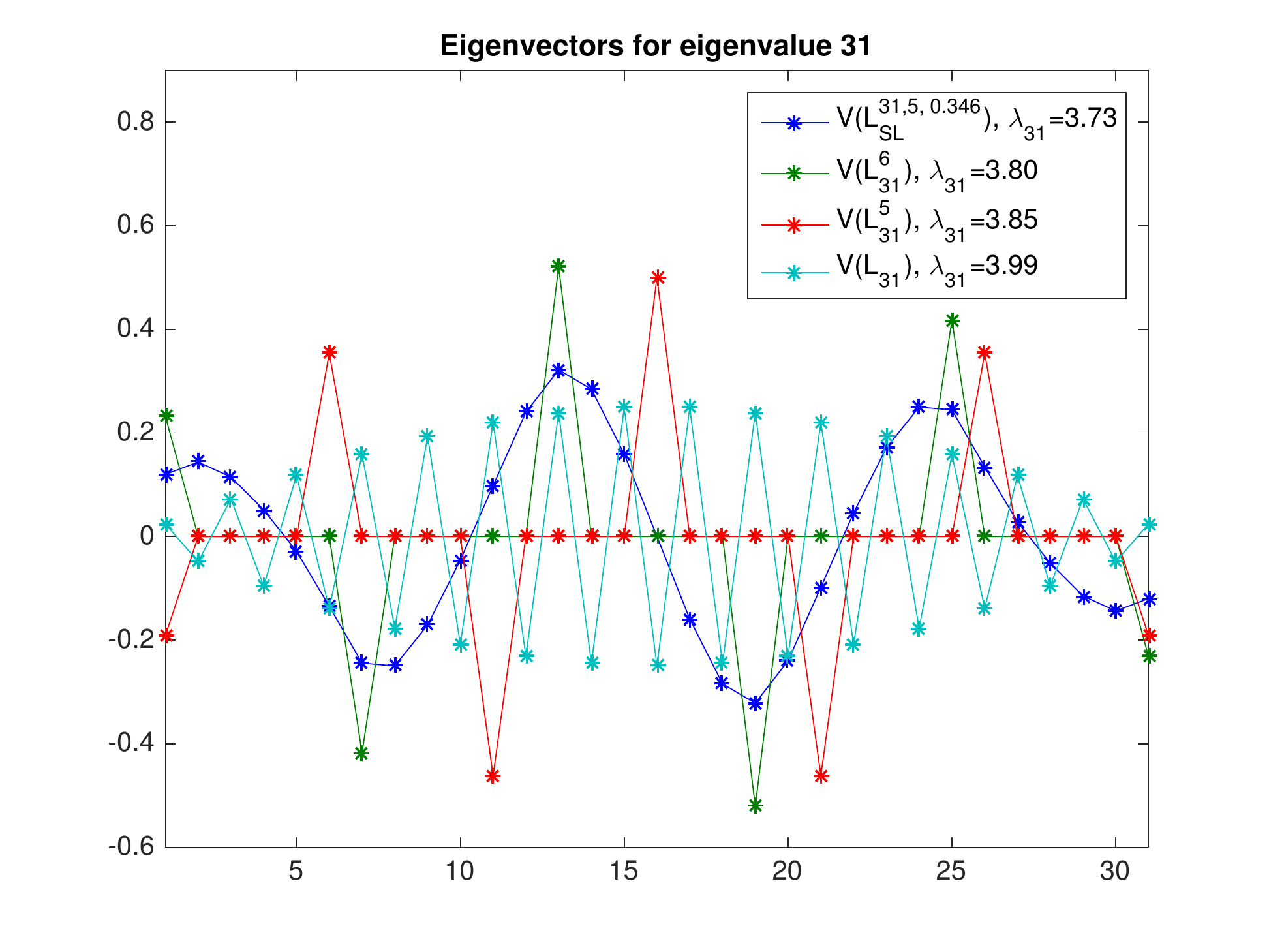}}
	\caption{Eigenvalues of $L^{N, m, \gamma}_{SL}$, $L^m_N$, $L^{m+1}_N$ and $L_N$ with parameter values $N = 31$, $m = 5$ and $\gamma = 0.346$ and the eigenvectors corresponding to three eigenvalues of the same matrices. 
	}
	\label{fig:eigvectors}
\end{figure}

The spectrum of higher-dimensional constant coefficient Laplacians can be inferred from the spectrum of the one-dimensional matrices by means of Kronecker products. Next, we consider the properties of common smoothers when applied to LISL discretization matrices of the two-dimensional Laplacian and conclude with an example illustrating the impact of the diffusion coefficient on the convergence of geometric multigrid cycles.

\subsection{Local Fourier analysis of the smoothers}

We seek to analyse how a varying size stencil affects the properties of the standard Gauss-Seidel smoother. 
We base the analysis on Local Fourier Analysis (LFA) as described in Chapter 4 of \cite{trottenberg2001multigrid} and state the \textit{smoothing factors} $\mu_{\text{loc}}$ of Gauss-Seidel iterations when applied to wide stencil finite difference discretizations. The key to the analysis is the use of grid functions of the form $\varphi(\bm{\theta}, \bm{x}) = e^{i \bm{\theta} \cdot \bm{x}}$, where $i$ is the imaginary unit, $\bm{x} \in \mathbb{R}^d$, $\bm{\theta} \in [-\pi, \pi)^d$ and $\cdot$ is the inner product for vectors in $\mathbb{R}^d$. For simplicity we consider equispaced grids $\Omega_{\Delta x}$ with refinement parameter $\Delta x > 0$. Therefore, any $\bm{x} \in \Omega_{\Delta x}$ can be written as $\bm{x} \equiv \bm{x}_0 + \kappa \Delta x$ for some fixed $\bm{x}_0 \in \Omega_{\Delta x}$ and $\kappa \in \mathbb{Z}^d$. It is thus convenient to rescale the exponent of $\varphi$ by $\Delta x^{-1}$.

The functions $\varphi$ are important since, as shown in Lemma 4.2.1 of \cite{trottenberg2001multigrid}, ``all grid functions $\varphi(\bm{\theta}, \bm{x})$ are (formal) eigenfunctions of any discrete operator which can be described by a difference stencil". This property allows us to associate to each discrete finite difference operator $L_{\Delta x}$ a so-called symbol $\tilde{L}_{\Delta x}(\bm{\theta})$ defined by
\begin{align}
\label{symbol}
L_{\Delta x} \varphi(\bm{\theta}, \bm{x}) = \sum_{\kappa \in \mathbb{Z}^d} s_{\kappa} e^{i\bm{\theta} \cdot \kappa} = \tilde{L}_{\Delta x}(\bm{\theta}) e^{i \bm{\theta} \cdot \kappa},
\end{align}
where $s_\kappa \in \mathbb{R}$ is the finite difference coefficient at the location $\kappa$ with respect to the node $\bm{x}_0$.

As in \cite{trottenberg2001multigrid}, we consider smoothers formed by a splitting $L_{\Delta x} = L_{\Delta x}^+ + L_{\Delta x}^-$ of the discrete operator, i.e.\
\[
S_{\Delta x} = \left(L^+_{\Delta x}\right)^{-1} L^-_{\Delta x}.
\]
Lemma 4.3.1 in \cite{trottenberg2001multigrid} derives the expression for the symbol for the smoother as
\[
\tilde{S}_{\Delta x}(\bm{\theta}) \defeq \frac{\tilde{L}^-_{\Delta x}(\bm{\theta})}{\tilde{L}^+_{\Delta x}(\bm{\theta})},
\]
where $\tilde{L}^+_{\Delta x}$ and  $\tilde{L}^-_{\Delta x}$ are defined as for $L_{\Delta x}$ in (\ref{symbol}).

With multigrid, the objective of the smoother is to dampen error components not reduced by the coarse grid correction. Therefore, assessing the properties of a given smoother requires fixing the coarse grid correction. We limit the study to the simplest coarsening strategy, that is if $\Omega_{\Delta x}$ is the fine grid then $\Omega_{2\Delta x}$ is the coarse grid. 
This leads to the definition of low and high frequencies below.

\begin{definition}[Definition 4.2.1 in \cite{trottenberg2001multigrid}]
\label{def:freqs} 
For the coarsening considered, we define the high and low frequencies as follows:
\begin{align*}
 \varphi(\bm{\theta},\cdot) \text{ low frequency component } & \Longleftrightarrow \bm{\theta} \in T^{\text{low}} \defeq \left[  -\frac{\pi}{2}, \frac{\pi}{2}\right)^d; \\ 
 \varphi(\bm{\theta},\cdot) \text{ high frequency component } &  \Longleftrightarrow \bm{\theta} \in T^{\text{high}} \defeq [-\pi, \pi)^d {\Big \backslash} \left[  -\frac{\pi}{2}, \frac{\pi}{2}\right)^d.
\end{align*}
\end{definition}


\begin{definition}[Definition 4.3.1 in \cite{trottenberg2001multigrid}]
The smoothing factor 
for standard coarsening is
\begin{equation*}
\mu_{\text{loc}}= \mu_{\text{loc}}(S_{\Delta x}) \defeq \sup \left\{ | \tilde{S}_{\Delta x}(\bm{\theta}) | : \bm{\theta} \in T^{\text{high}} \right\}.
\end{equation*}
\end{definition}

We employ these definitions to compare the smoothing factors for the standard two-dimensional Laplacian, setting $d = 2$, discretised using standard local finite differences and the LISL discretization.

\begin{example}[Example 4.3.4  in \cite{trottenberg2001multigrid}]
\label{ex:mu_GS_STD}
The smoothing factor for the Gauss-Seidel smoother for the standard Laplacian discretisation
%
%
%
is given by
\begin{equation*} 
\mu_{\text{loc}} = \sup \left\{ \left| \frac{e^{i \theta_1} + e^{i \theta_2}  }{4 - e^{-i \theta_1} - e^{-i \theta_2}} \right| : \bm{\theta} \in T^{\text{high}} \right\}.
\end{equation*}
\end{example}

Similarly, the smoothing factor for the LISL scheme can be derived. In the present case of pure diffusion, Schemes 1--3 coincide.

\begin{example} Proceeding as in \cite{trottenberg2001multigrid} for Example \ref{ex:mu_GS_STD}, the symbols $L^+_{\Delta x}$ and $L^-_{\Delta x}$ for the LISL discretizations are
\begin{align*}
\tilde{L}^+_{\Delta x}(\bm{\theta}) &= \frac{1}{\Delta x} (4 - \gamma_1 e^{-i m_1 \theta_1} - (1-\gamma_1) e^{-i (m_1 + 1) \theta_1} - \gamma_2 e^{-i m_2 \theta_2} - (1-\gamma_2) e^{-i (m_2 + 1) \theta_2} ),  \\
\tilde{L}^-_{\Delta x}(\bm{\theta}) &=  - \frac{1}{\Delta x}(\gamma_1 e^{i m_1 \theta_1} + (1-\gamma_1) e^{i (m_1 + 1) \theta_1} + \gamma_2 e^{i m_2 \theta_2} + (1-\gamma_2) e^{i (m_2 + 1) \theta_2} ),
\end{align*}
where $m_i$ and $\gamma_i$ are given by \eqref{eq:def_mAlpha} replacing $\sigma$ by $\sigma_i$.
For compactness of notation, define
\begin{equation*}
g(\theta, \gamma, m) \defeq \gamma e^{i m \theta} + (1-\gamma) e^{i (m+1) \theta},
\end{equation*}
then the smoothing factor for a Gauss-Seidel smoother with standard coarsening and the LISL scheme is given by
\begin{equation} \label{eq:smoothingfactorGS}
\mu_{\text{loc}}(S_{\Delta x}^{SL}) = \sup \left\{ \left| \frac{ g_1 + g_2  }{4 - \bar{g_1} - \bar{g_2}} \right| : \bm{\theta} \in T^{\text{high}} \right\},
\end{equation}
for $\bm{\theta} \in T^{\text{high}}$, where $g_1 \equiv g(\theta_1, \gamma_1, m_1)$, $g_2 \equiv g(\theta_2, \gamma_2, m_2)$, and $\bar{c}$ denotes the complex conjugate of the complex number $c$.
\end{example}

From \eqref{eq:smoothingfactorGS} we see that as the non-locality of the discretization grows, i.e.\ $m_. \to \infty$ then the smoothing factor approaches 1 (no smoothing) and so highly oscillatory modes will be transferred to the coarser subspace. Figure \ref{fig:LFA_GS} compares the smoothing factor for the fixed stencil 5 point discretization and a specific semi-Lagrangian stencil.

\begin{figure}[htp]
	\centering
	\subfloat[Smoothing factor Gauss-Seidel applied to standard FD approximation of the Laplacian.]
	{\label{subfig:lfa_std_GS_2D}\includegraphics[width=0.49\textwidth]
	{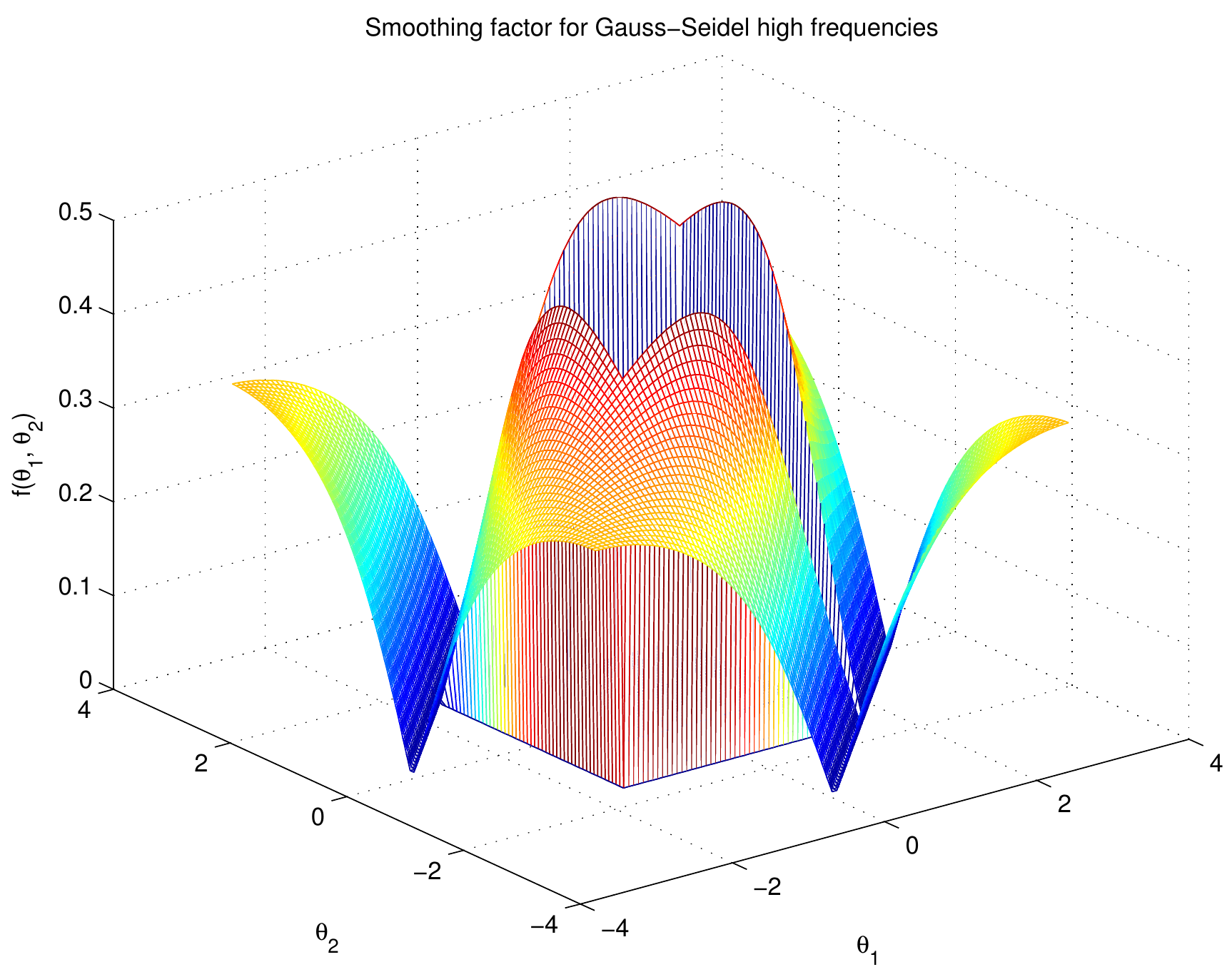}}
	\hspace{\stretch{1}}
	\subfloat[Smoothing factor Gauss-Seidel applied to LISL approximation of the Laplacian with stencil parameters $((m_1 = 9, m_2 = 3), (\gamma_1 = 0.5, \gamma_2 = 1)) $.]
	{\label{subfig:lfa_LISL_GS_2D}\includegraphics[width=0.49\textwidth]
	{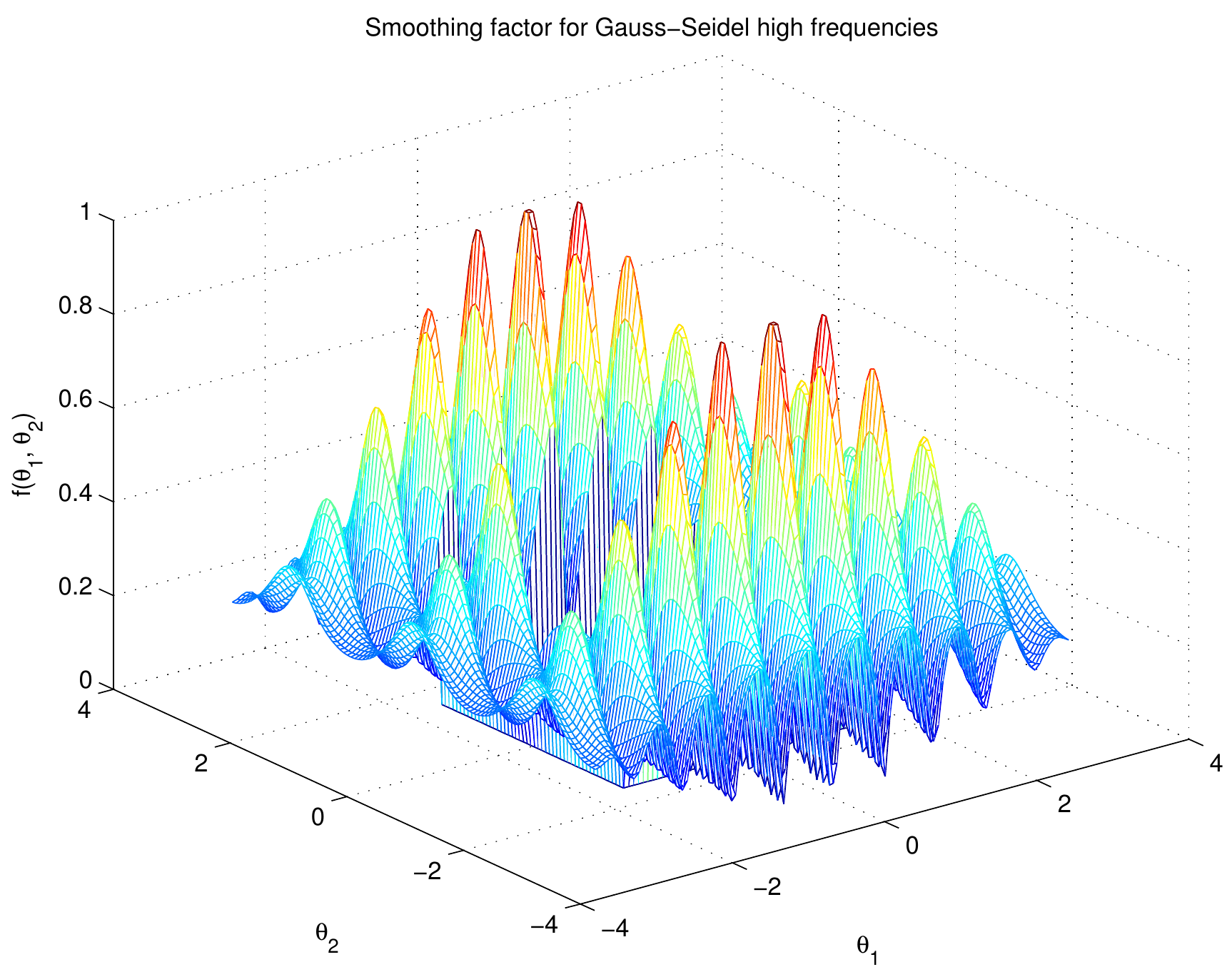}}
	\caption{Representation of the smoothing factor for high frequencies, i.e.\ $\mathbf{\theta} \in [-\pi, \pi]^2 \setminus	[-\pi /2, \pi /2]^2$, for the Gauss-Seidel iteration for the classical fixed stencil Finite Difference (FD) and the LISL schemes of the two-dimensional Laplacian operator. The maxima calculated numerically are 0.49 (theoretical value is 0.5) for the fixed stencil FD and 0.95 for the LISL scheme (lower is better).}
	\label{fig:LFA_GS}
\end{figure}

\begin{example} We can generalise the results in the previous example to the case of diffusion given by a vector $(\sigma_1, \sigma_2)^T$ not necessarily parallel to any of the axes. If $\sigma_1$ and $\sigma_2$ have the same sign, then
\begin{align*}
\tilde{L}^+_{\Delta x}(\bm{\theta}) = \frac{1}{\Delta x} (2 - \bar{g}(\theta_1, \gamma_1, m_1) \bar{g}(\theta_2, \gamma_2, m_2) ),  \;\;\;\;\; 
\tilde{L}^-_{\Delta x}(\bm{\theta}) =  - \frac{1}{\Delta x}(g(\theta_1, \gamma_1, m_1) g(\theta_2, \gamma_2, m_2) ).
\end{align*}
If, however, $\sigma_1$ and $\sigma_2$ have different signs, then
\begin{align*}
\tilde{L}^+_{\Delta x}(\bm{\theta}) = \frac{1}{\Delta x} (2 - g(\theta_1, \gamma_1, m_1) \bar{g}(\theta_2, \gamma_2, m_2) ),  \;\;\;\;\; 
\tilde{L}^-_{\Delta x}(\bm{\theta}) =  - \frac{1}{\Delta x}(\bar{g}(\theta_1, \gamma_1, m_1) g(\theta_2, \gamma_2, m_2) ).
\end{align*}
To account for the fact that $\sigma_\cdot$ can be negative, we re-define $m_i \defeq \left\lfloor |\sigma_i|/\sqrt{\Delta x} \right\rfloor$ and $\gamma_i \defeq  (m_i+1) -|\sigma_i|/\sqrt{\Delta x}$. The deterioration of the smoother for large $m_i$ is present here too. 
\end{example}

\subsection{Performance of geometric multigrid}
\label{sec:NE_GMG}

We conclude the discussion of geometric multigrid by testing its performance against
an iterative solver used in \cite{ma2014unconditionally}, i.e.\ BICGSTAB \cite{BICGSTAB} with and without ILU(0)\footnote{Incomplete LU factorization with the same sparsity pattern as the original system matrix.} as preconditioner \cite{SaadIterative}, 
and algebraic multigrid algorithms, namely,
the classical Ruge-St{\"u}ben AMG \cite{ruge1987algebraic} using our own implementation, and AGMG from \cite{notay2012aggregation}, using the implementation from \cite{notayweb}.

As benchmark examples, we choose a linear system $Ax = b$ whose coefficient matrix is the LISL discretization of \eqref{eq:laplacian} in the two-dimensional square $[0, 1]^2$ with Dirichlet boundary conditions, and $\sigma = 2 I_2$ and $\sigma = \sqrt{5} I_2$, respectively, where $I_2$ is the $2\times 2$ identity matrix. 
These values are chosen to study the effect of interpolation (which is always required in the second case and only for odd levels in the first) on the convergence and complexity of the methods, in particular on the convergence rates of geometric multigrid and the operator complexity of algebraic multigrid.

We use a Cartesian grid with equal number of equispaced nodes in both directions, the smoother is Gauss-Seidel, the prolongation operator bilinear interpolation, the restriction the transpose of the prolongation, and the coarse grid  operator is constructed using the Galerkin principle. 

Figure \ref{fig:geometric_GS} presents the reduction of the residual, $r_k \equiv \|b - Ax^k\|_2$, against the number of iterations $k$ with $x^0 = 0$, for a discrete mesh where the distance between two consecutive nodes is $\Delta x = 2^{-8}$. The algorithm is stopped whenever the relative residual, $\| b - A x^k\|_2/\|b\|_2$, measured in the Euclidean norm, is below the prescribed tolerance, in this case $10^{-6}$.

\begin{figure}[ht]
	\subfloat[Residual after the $k$-th iteration for $\sigma = 2 I_2$.]
	{\label{subfig:MG_2}\includegraphics[width=0.49\textwidth]
	{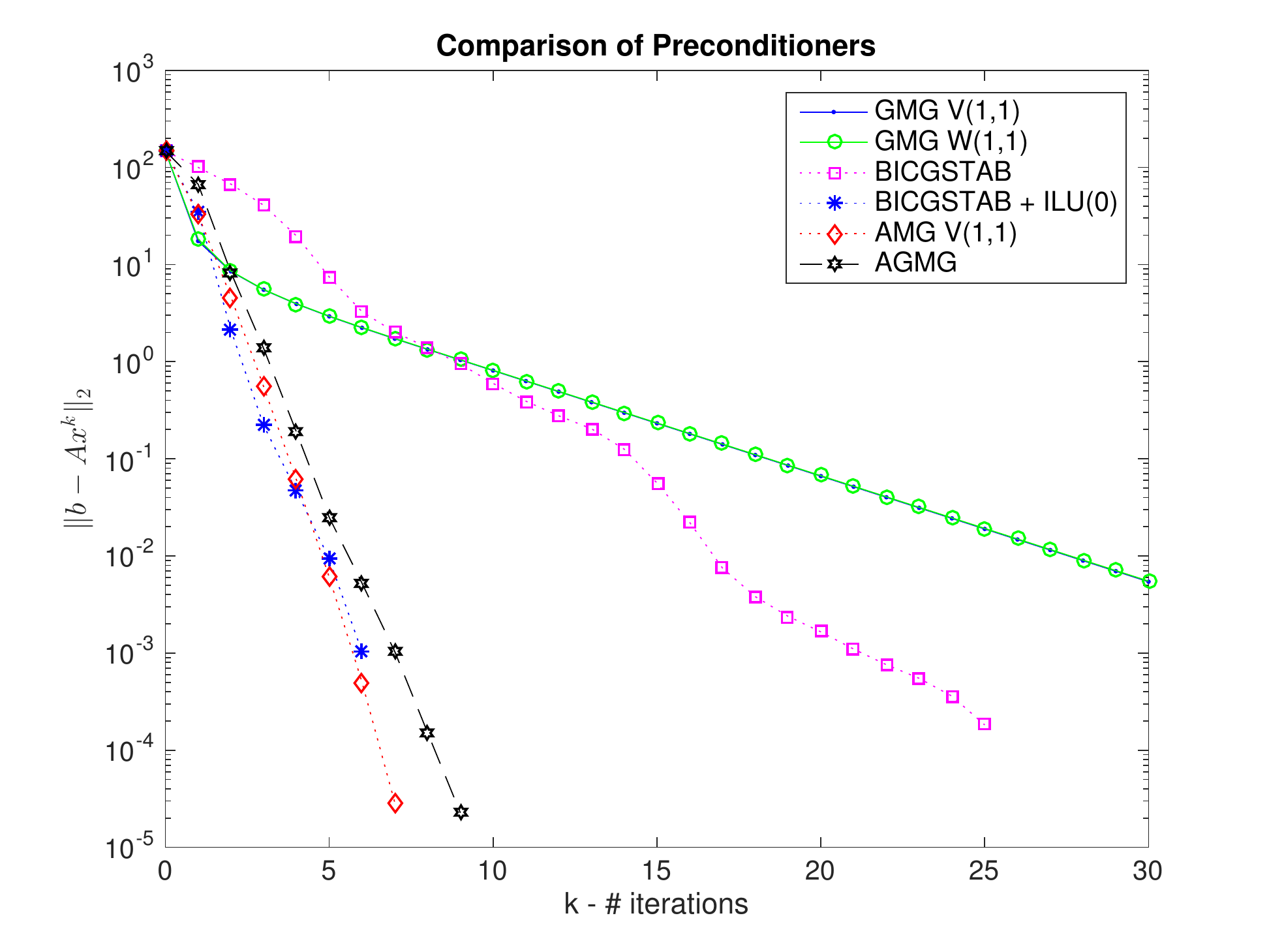}}
	\hspace{\stretch{1}}
	\subfloat[Residual after the $k$-th iteration for $\sigma = \sqrt{5} I_2$.]
	{\label{subfig:MG_sqrt5}\includegraphics[width=0.49\textwidth]
	{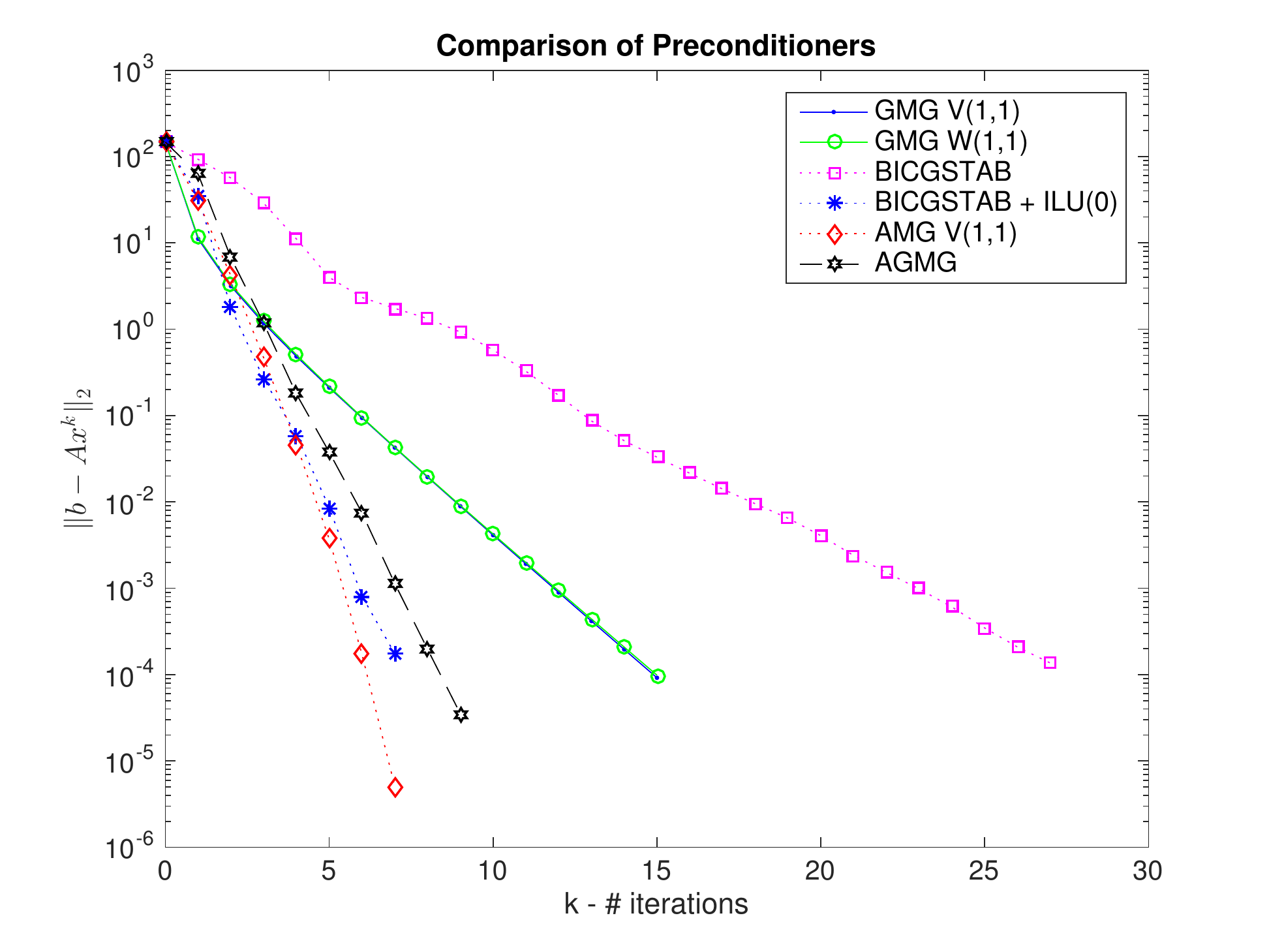}}
	\caption{Residual $\|b - Ax^k\|_2$ in the Eucledian norm at the end of the $k$-th iteration of different geometric and algebraic multigrid cycles when solving \eqref{eq:laplacian} on equispaced Cartesian grid of $[0, 1]^2$ with 257 nodes per dimension and with homogeneous Dirichlet boundary conditions. Geometric $V(\nu_1, \nu_2)$ and $W(\nu_1, \nu_2)$ cycles are considered, where $\nu_1$ and $\nu_2$ denote the number of pre- and post-smoothing steps. Their performance is compared to the iterative method BICGSTAB with and without preconditioner, and to two algebraic algorithms, AMG and AGMG from \cite{ruge1987algebraic} and
	\cite{notay2012aggregation}, respectively (see also Sections \ref{subsec:agmg} and \ref{subsec:perform}). Notice the almost overlapping of lines for geometric $V(\nu_1, \nu_2)$ and $W(\nu_1, \nu_2)$ cycles for equal $\nu_1$ and $\nu_2$ (see also Table \ref{Tab:rhoMG}, which shows almost identical rates for $\ell = 8$).}
	\label{fig:geometric_GS}
\end{figure}

Next, we study the residual reduction factor $\rho = (r_k/r_0)^{1/k}$, where $k$ is the number of iterations required for the prescribed tolerance.
We solve the problems above for different  refinement levels $\ell$, where the number of nodes per dimension is $2^\ell$ +1. We observe 
in Table \ref {Tab:rhoMG} that for even $\ell$ the convergence factor $\rho$ corresponding to geometric multigrid cycles for $\sigma = 2 I_2$ is significantly worse than that for $\sigma = \sqrt{5} I_2$. This is due to the lack of interpolation when $2^\frac{\ell}{2} \in \mathbb{N}$, as the
step
is a multiple of $\Delta x$, so for any mesh node $x_{\Delta x} \in \Omega_{\Delta x}$ $x_{\Delta x} \pm \sqrt{\Delta x} \sigma_i \in \Omega_{\Delta x}$. The lack of interpolation, and the equal stencil lengths in both directions, gives $\mu_{\text{loc}} = 1$ in (\ref{eq:smoothingfactorGS}). Moreover, as shown in Figure \ref{fig:eigvectors}, the eigenvectors corresponding to small eigenvalues are highly oscillatory and hence not resolved sufficiently on the coarse mesh. 

Regarding the BICGSTAB iterative solver, we observe the benefit of using ILU(0) as preconditioner, however, the significant increase in the convergence rate as the mesh is refined (and hence the condition number of the matrix increases) suggests that convergence is not asymptotically mesh size independent.
To further illustrate this, Table \ref{Tab:rhoMG1D} contains the residual reduction factors for AGMG and BICGSTAB with and without preconditioner when solving \eqref{eq:laplacian} in the one-dimensional domain $[0, 1]$ with homogeneous Dirichlet boundary conditions and discretized using the LISL scheme. We consider the cases $\sigma= 2$ and $\sigma=\sqrt{5}$. As discussed previously, for $\sigma = 2$ and $\ell$ even, no interpolation is required. In this case, ILU(0) exactly factorises the system matrix $A$ and hence $\rho = 0$. 
However, when interpolation is needed, $\rho$ approaches 1 for BICGSTAB as the mesh is refined but not for AGMG. 
Furthermore, for $\ell = 21$ BICGSTAB and BICGSTAB with ILU(0) require approximately 86 and 21 times the time required by AGMG for the relative residual to be below $10^{-6}$. 
Additionally, let $t_{tol}$ be the time required for the relative residual to be below $10^{-6}$ and $N$ the number of unknowns, assuming that $t_{tol} \sim \mathcal{O}(N^a)$, empirically we observe that $a$ equals 1.11 for AGMG, 1.50 for BICGSTAB and 1.36 for BICGSTAB with ILU(0) as preconditioner.

\begin{table}[p]\captionsetup[subfloat]{position=top}
\begin{center}
\subfloat[$\rho$ for $\sigma = 2I_2$]{
\begin{tabular}{cc|c|c|c|c|c|c}
\multirow{2}{*}{$\ell$}& \multirow{2}{*}{$m$} & \multirow{2}{*}{GMG V(1,1)}&\multirow{2}{*}{GMG W(1,1)}&\multirow{2}{*}{AMG}& \multirow{2}{*}{AGMG} & \multirow{2}{*}{BICGSTAB} & BICGSTAB \\ 
&  & & & & & & with ILU(0) \\
\hline
6  & 16 & 0.4193  &  0.4204  &    0.0415  &  0.1015 & 0.2502 & 0.0151 \\
7  & 22 & 0.2612  &  0.2666  &    0.0633  &  0.1502 & 0.5158 & 0.0767 \\
8  & 32 & 0.7561  &  0.7564  &    0.0981  &  0.1551 & 0.5763 & 0.1234 \\
9  & 45 & 0.5076  &  0.4905  &    0.1216  &  0.1858 & 0.7109 & 0.2392 \\
10 & 64 & 0.8823  &  0.8841  &    0.1219  &  0.2001 & 0.7621 & 0.3382 
\end{tabular}
}
\\
\subfloat[$\rho$ for $\sigma = \sqrt{5}I_2$]{
\begin{tabular}{cc|c|c|c|c|c|c}
\multirow{2}{*}{$\ell$}& \multirow{2}{*}{$m$} & \multirow{2}{*}{GMG V(1,1)}&\multirow{2}{*}{GMG W(1,1)}&\multirow{2}{*}{AMG}& \multirow{2}{*}{AGMG} & \multirow{2}{*}{BICGSTAB} & BICGSTAB \\ 
&  & & & & & & with ILU(0) \\
\hline
6	& 17 & 0.2999  &  0.2857  & 0.0403  &  0.1162 & 0.3718  &  0.0262 \\
7 	& 25 & 0.2565  &  0.2568  & 0.0532  &  0.1367    & 0.4408  &  0.0302 \\
8 	& 35 & 0.4348  &  0.4300  & 0.0740  &  0.1656    & 0.5938  &  0.1302 \\
9	& 50 & 0.4480  &  0.4314  & 0.1124  &  0.2030    & 0.6751  &  0.1746 \\
10  & 71 & 0.5547  &  0.4799  & 0.1272  &  0.1992    & 0.7478  &  0.3374
\end{tabular}
}
\end{center}
\caption{The residual reduction factor $\rho$ for different mesh sizes and different multigrid algorithms, for the two-dimensional Laplace equation; 
the length of the stencil $m$ as per \eqref{eq:def_mAlpha}.\label{Tab:rhoMG}}
\end{table}

\begin{table}[p]\captionsetup[subfloat]{position=top}
\begin{center}
\subfloat[$\rho$ for $\sigma = 2$]{
\begin{tabular}{c|c|c|c}
\multirow{2}{*}{$\ell$} & \multirow{2}{*}{AGMG} & \multirow{2}{*}{BICGSTAB} & BICGSTAB \\ 
& & & with ILU(0) \\
\hline
10  &   0.2486  &  0.6445  &       0 \\
15  &   0.4680  &  0.9479   &  0.7105 \\
20  &   0.5291  &  0.9815  &       0 \\
21  &   0.6524  &  0.9935  &  0.9735
\end{tabular}
}
\subfloat[$\rho$ for $\sigma = \sqrt{5}$]{
\begin{tabular}{c|c|c|c}
\multirow{2}{*}{$\ell$} & \multirow{2}{*}{AGMG} & \multirow{2}{*}{BICGSTAB} & BICGSTAB \\ 
& & & with ILU(0) \\
\hline
10  &       0.3298  &  0.7515   & 0.3079 \\
15  &       0.5640  &  0.9312  &  0.7703 \\
20  &       0.6347  &  0.9890   & 0.9550 \\
21  &       0.4780  &  0.9940    & 0.9617
\end{tabular}
}
\end{center}
\caption{Comparison of the residual reduction factor $\rho$ for different system sizes and different solvers for the one dimensional Laplace equation. The system size is $2^\ell +1$.\label{Tab:rhoMG1D}}
\end{table}

\begin{table}[p]\captionsetup[subfloat]{position=top}
\begin{center}
\subfloat[Complexities for $\sigma = 2I_2$]{
\begin{tabular}{c|cc|cc|cc}
\multirow{2}{*}{$\ell$}&\multicolumn{2}{c|}{GMG}&\multicolumn{2}{c|}{AMG}&\multicolumn{2}{c}{AGMG}\\ 
\cline{2-7}
&$c_G$&$c_A$&$c_G$&$c_A$&$c_G$&$c_A$\\\hline
6 & 1.31 & 3.66  & 1.75 & 1.74  & 1.24   & 1.18  \\
7 & 1.32 & 2.69  & 1.76 & 6.92  & 1.26   & 1.26  \\
8 & 1.33 & 3.90  & 1.72 & 2.05  & 1.33   & 1.28  \\
9 & 1.33 & 2.75  & 1.71 & 11.79 & 1.25   & 1.31  \\
10 & 1.33 & 3.97  & 1.70 & 2.16 & 1.25   & 1.23 
\end{tabular}
}
\subfloat[Complexities for $\sigma = \sqrt{5} I_2$]{
\begin{tabular}{c|cc|cc|cc}
\multirow{2}{*}{$\ell$}&\multicolumn{2}{c|}{GMG}&\multicolumn{2}{c|}{AMG}&\multicolumn{2}{c}{AGMG}\\ 
\cline{2-7}
&$c_G$&$c_A$&$c_G$&$c_A$&$c_G$&$c_A$\\\hline
6 & 1.31 & 2.59  & 1.67 & 3.78  & 1.18   & 1.10 \\
7 & 1.32 & 2.69  & 1.74 & 6.48  & 1.24   & 1.22 \\
8 & 1.33 & 2.65  & 1.71 & 7.47  & 1.20   & 1.22 \\
9 & 1.33 & 2.76  & 1.69 & 9.34  & 1.22   & 1.30 \\
10 & 1.33 & 2.67  & 1.64 & 7.39 & 1.32   & 1.45 
\end{tabular}}
\end{center}
\caption{Comparison of the grid and algebraic complexities as per Definitions \ref{gridc} and \ref{algebrac} for different mesh sizes and different multigrid algorithms, for the two-dimensional case.\label{Tab:compMG}}
\end{table}

Returning to the two-dimensional case, the grid hierarchies in the geometric (GMG) and algebraic (AMG) multigrid have 5 levels, including the finest one. In the geometric case, the number of unknowns is $4$ times smaller from one level to the next. In the AGMG case, the hierarchy is at most 4 levels deep. Table \ref{Tab:compMG} reports the complexities for the three categories of algorithms considered. The results confirm the assertion in \cite{wathen2015precond} that AMG coarsening, generally, need not reduce the number of unknowns fast enough. In the present setting, contrasting the case $\sigma = \sqrt{5} I_2$ with $\sigma = 2I_2$ shows that the growth in complexity is due to the interpolation, which creates
a denser connectivity graph on the coarser levels. 



\subsection{Properties of the LISL matrix}
\label{subsec:agmg}

In this section, we discuss the theoretical foundation of \emph{Aggregation-based Multigrid} (AGMG) for our specific application of wide stencil discretisations.

The key result is Lemma 3.1 in \cite{notay2012aggregation}. 
The non-negativity of the row sums of a LISL discretization matrix is obtained almost by construction. To see this, let $A \in \mathbb{R}^{N\times N}$ be the discretization matrix, where $N \defeq |\Omega_{\Delta x}|$ is the number of mesh points, then the sum for the $i$-th row is
\begin{align*}
\sum_{j=1}^N (A)_{ij} 
&= 1 + \Delta t \left( \frac{M}{\Delta x} - c^{\alpha, n}_i \right) - \frac{\Delta t}{\Delta x} M \geq 0,
\end{align*}
where we have used the fact that for any $z \in \bar{\Omega}$, $\sum_{j=1}^N w_j(z) = 1$ and the CFL-type condition $1 - \Delta t c^{\alpha, n}_i \geq 0$, which is satisfied for sufficiently small $\Delta t$ independent of $\Delta x$.

The following analysis of the non-negativity of the column sum makes use of the regularity of the coefficients $b$ and $\sigma$. 
In particular, we
assume that the coefficients are such that for all $p \in [[1, P]]$, and for any mesh points $x_i, x_l$ and corresponding controls $\alpha_i, \alpha_l \in \mathcal{A}$ and $s \in [0, T]$ we have that
\begin{eqnarray}
\label{sigreg}
\| \sigma^{\alpha_i}_p(s, x_i) - \sigma^{\alpha_l}_p(s, x_l) \|_\infty \leq  L_\sigma \| x_i-x_l \|^\eta_\infty, \;\;\;
\| b^{\alpha_i}(s, x_i) - b^{\alpha_l}(s, x_l) \|_\infty \leq L_b \| x_i-x_l \|^\beta_\infty,
\label{breg}
\end{eqnarray}
where $\beta \in \left(0, 1\right]$, $\eta \in \left( \frac{1}{2}, 1\right]$. 

\begin{remark}
As stated in the introduction, we are working under the standard assumption of Lipschitz continuity of the coefficients in $x$ and continuity in $\alpha$. However,
what we require in (\ref{sigreg}) 
is stronger, namely, if the control is inserted in the coefficients as a function of the state $x$, the resulting functions  are H{\"o}lder continuous in $x$. This situation arises in every step of the policy iteration algorithm: a control vector $(\alpha_i)$ is determined by the optimisation step, and then a linear system with this control vector is solved for $(x)_i$. Generally, the optimal control is not a (H{\"o}lder) continuous function of the space variables, but there are many important examples where it is at least piecewise H{\"o}lder. It can be seen from the proof below that Proposition \ref{prop-columns} still holds in this situation.
\end{remark}
\begin{remark}
Lemma 3.1 in \cite{notay2012aggregation}
also assumes that the system matrix is irreducible. LISL discretization matrices need not be irreducible, e.g.\ $L_{SL}^{2K, 2, 1}$ for any $K \in \mathbb{N}$ as in \eqref{eq:LSL}, however, this technical requirement could be overcome by adding an irreducible M-matrix, multiplied by a sufficiently small factor, to the LISL discretization matrix.
\end{remark}


We also assume the use of multi-linear interpolation requiring $2^d$ points to approximate function values in $\mathbb{R}^d$ and the use of Cartesian grids. 


\begin{proposition}
\label{prop-columns}
Let $A$ 
be the LISL discretization matrix of \eqref{eq:1.1} 
for a given time step, on an equispaced Cartesian grid $\Omega_{\Delta x}$ of $\Omega \subset \mathbb{R}^d$ with $\Delta x > 0$, and
a given vector of control values $(\alpha_i)_{i=1,\ldots,N}$, $\alpha_i \in \mathcal{A}$, associated with the mesh points $x_i$, $1\le i \le N:=|\Omega_{\Delta x}|$.
Assume that (\ref{sigreg}) 
holds.

Then the column sum of the matrix is non-negative provided
\begin{align} \label{eq:columnBoundHolder}
\Delta t \leq \frac{\Delta x}{ \sup_{\alpha\in \mathcal{A}}{|c^{\alpha, +}|} + (\mathcal{M}-1) (P+1) },
\end{align}
where $\mathcal{M}$ depends on the dimension of the domain $d$ and the Lipschitz constants $L_\sigma$ and $L_b$, but not on the mesh parameter $\Delta x$. Indeed, $\mathcal{M}= 3^d$ for sufficiently small $\Delta x$.
%
\end{proposition}

\begin{proof}
We carry out the proof for Scheme 2, 
but an analogous analysis holds for Schemes 1 and 3 in the introduction.
We also note that we can restrict the analysis to steps where no truncation is required, as in the case of truncation the weights only contribute (positively) to the diagonal of the matrix and the right-hand side of the equation (see Remark \ref{rem:rhs}).

For simplicity of notation, we omit the dependence of the coefficient functions $b$ and $\sigma_p$ on the time variable $t$ and the control. For any $i \neq j$ the matrix entry $(A)_{ij} \neq 0$ if and only if for any $1 \leq m \leq P+1$ we require $\phi(x_j)$ to approximate -- by means of linear interpolation -- $\phi(x_i + y^\pm_m(x_i))$, where $y^\pm_m(x_i)$ is either $y^\pm_p(x_i) = \pm \sqrt{\Delta x} \sigma_p(x_i)$ for $1 \leq p \leq P$ or $y^\pm_{P+1}(x_i) = \Delta x b(x_i)$. 
For any two nodes $i$ and $l$ to contribute to the sum of column $j$, it is necessary that
\begin{align*}
\| x_l + y^\pm_m(x_l) - (x_i + y^\pm_m(x_i)) \|_\infty < 2 \Delta x.
\end{align*}
As $x_l$ and $x_i$ lie on the grid, there exists a positive constant $M$ such that $\| x_l  - x_i \|_\infty = M \Delta x$. Then $(M+1)^d$ constitutes an upper bound on the number of terms the step $y^\pm_m(\cdot)$ contributes to the sum of column $j$.

We consider the different possible values for $y^\pm_m(x_l)$ separately. First, assume $y^\pm_m(x_l) = \Delta x b(x_l)$ and let $\Delta x \leq 1/(L_b \sqrt{d})$, then
\begin{align}
2 \Delta x > \| x_l + \Delta x b(x_l) - (x_i + \Delta x b(x_i)) \|_\infty 
\geq M \Delta x - \Delta x^{1+\beta} L_b M^\beta, \label{eq:M_bound_b} 
\end{align}
where we have used the triangle inequality and the H{\"o}lder regularity of $b$. Re-arranging, 
\begin{align*}
M < \frac{2}{1 -  L_b M^{\beta-1} \Delta x^\beta },
\end{align*}
such that
$M \le 2$ for sufficiently small $\Delta x$. 
Proceeding similarly for $y^\pm_m(x) = \pm \sqrt{\Delta x} \sigma_m(x)$, we obtain
%
%
again $M\le 2$ as $\Delta x \to 0$.

Denote $\mathcal{M}$ to be the maximum of all of the $(M+1)^d$s above, i.e.\ for different mesh points, then the column sum gives
\begin{align} \label{eq:bound}
\sum_{i=1}^N (A)_{ij} &= 1 + \Delta t \left( \frac{P+1}{\Delta x} - c^{\alpha, n}_j \right) - \frac{\Delta t}{2 \Delta x} \sum_{i=1}^N \sum_{p = 1}^{P+1} w_j(x_i + y_p(x_i)) \nonumber \\
&\geq 1 - \frac{\Delta t}{\Delta x}  \left(  c^{\alpha, n}_j + (\mathcal{M}-1) (P+1) \right).
\end{align}
Therefore, non-negativity of the sum is guaranteed by condition (\ref{eq:columnBoundHolder}).

\end{proof}

\begin{remark}
For the LISL scheme to be first order accurate, it is required that $\Delta t \sim \mathcal{O}(\Delta x)$. Therefore, the bound \eqref{eq:columnBoundHolder} does not impose problematic restrictions on the size of the time steps.
We recall that 
$\Delta t = \mathcal{O}( \Delta x)$ and
$\Delta t = \mathcal{O}( \Delta x^{3/2})$ (or even $\Delta t = \mathcal{O}( \Delta x^{2})$) are the CFL conditions for the explicit schemes without and with truncation, respectively, see Corollary \ref{cor:CFL}.
Therefore, on bounded domains, fully implicit time stepping with policy iteration and
AGMG preconditioning is the computationally most efficient overall algorithm among the ones considered.
\end{remark}

\subsection{Performance of the algebraic approaches}
\label{subsec:perform}

We compare the performance of the classical AMG implementation in the HSL library \cite{HSL}, and AGMG from \cite{notayweb} for the benchmark optimal control problems in Section \ref{pr:ProblemB}. Both of these methods are used as preconditioners for a Krylov subspace method that is assumed to have converged when the relative residual is below $10^{-6}$. 
In particular, we use MATLAB's implementation of GMRES \cite{SaadGMRES} for AMG and GCR \cite{GCR} for AGMG.
The AMG preconditioner consists of one iteration of the standard V-cycle with two Gauss-Seidel pre- and post-smoothing steps, whereas AGMG uses one Gauss-Seidel pre- and post-smoothing step and the enhanced multigrid cycles mentioned in the introduction, see \cite{notay2010aggregation, notay2012aggregation}.
For completeness, we also include as benchmark MATLAB's sparse direct solver using UMFPACK \cite{DavisUMFPACK}. The problems considered have smooth closed form solutions linear in $t$. As mentioned in the previous section, we employ policy iteration to solve the resulting non-linear discrete problem.
The tests were run on a Linux machine under MATLAB 2015a, on a quad-core AMD 4.2GHz with 7.5GB of RAM and 15GB of swap.

\begin{figure}[h]
	\centering
	\subfloat[
	Total time on solver for Problem A.]
	{\includegraphics[width=0.49\textwidth]{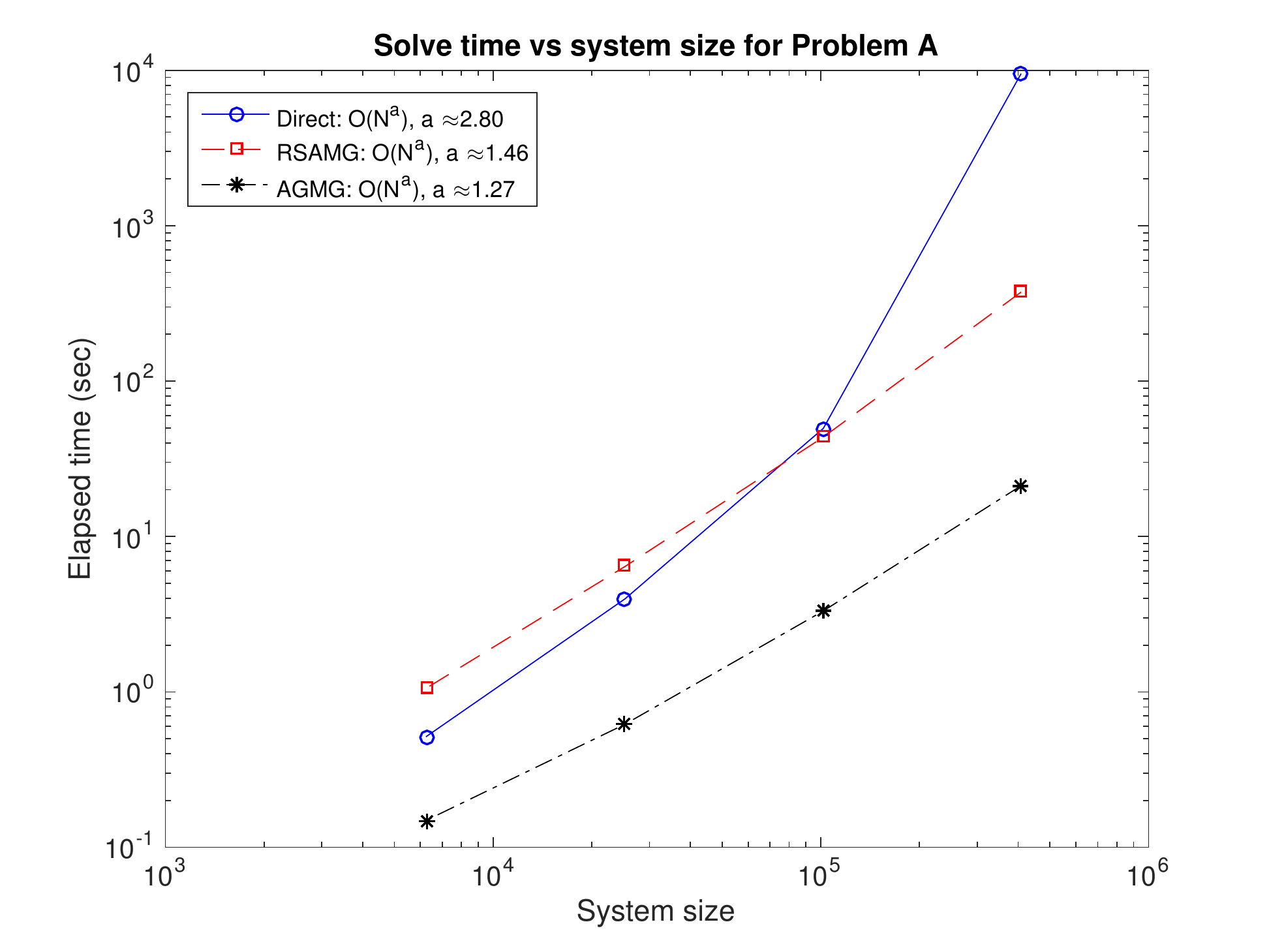}}
		\hspace{\stretch{1}}
		\subfloat[
	Total time on solver for Problem B.]
	{\includegraphics[width=0.49\textwidth]{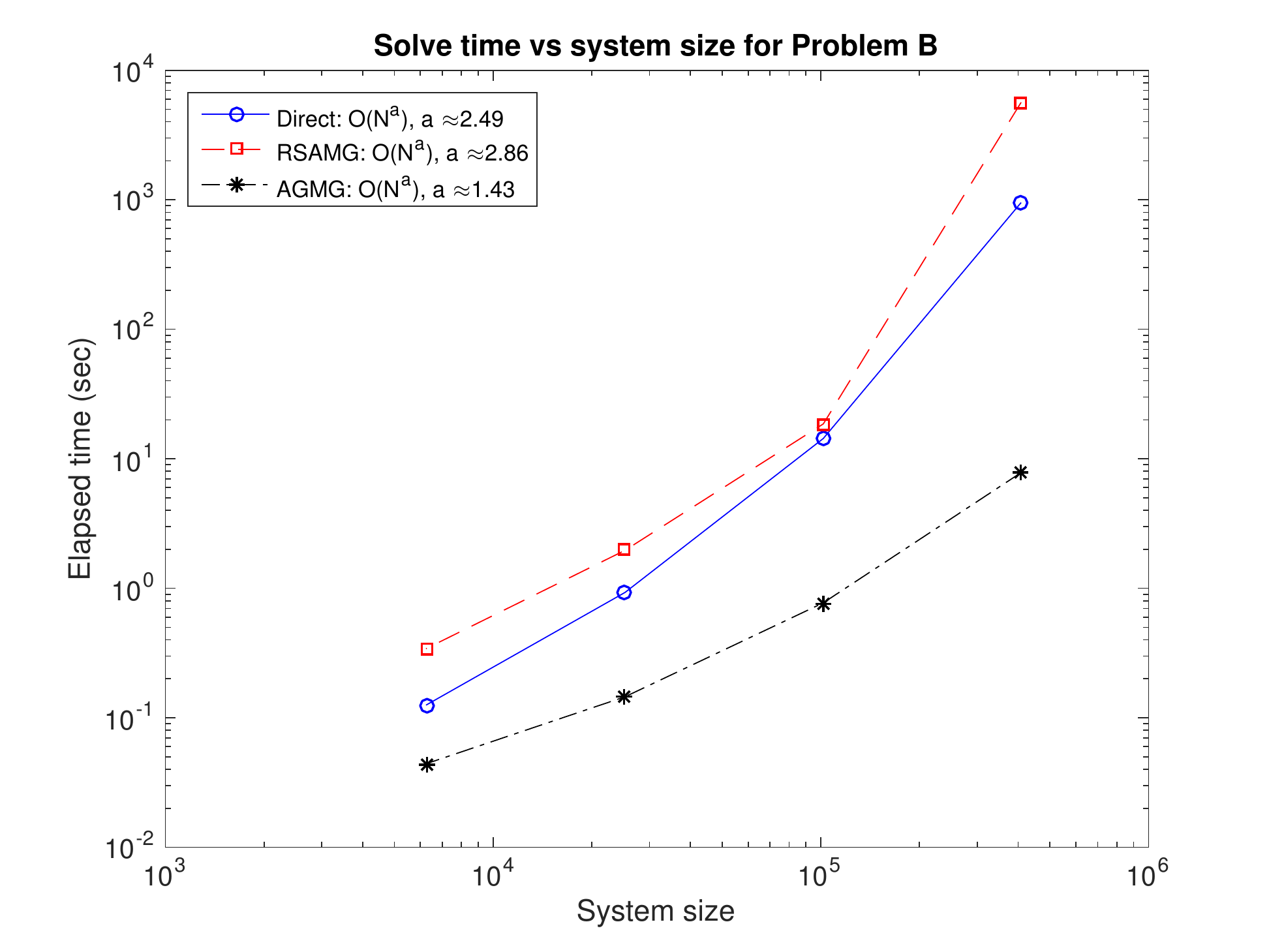}}
	\caption{Total number of seconds for solving the linear systems versus the size of the systems  for each of the linear system solvers considered. We use equispaced Cartesian grids in space with $81, 161, 321$ and $641$ nodes per dimension and one time step. }
	\label{fig:93B_errors_timing}
\end{figure}
In Figure \ref{fig:93B_errors_timing} we present the elapsed time solving linear systems for a single time step ($\Delta t = T$). Both MG methods provide a solution with the same accuracy as the sparse direct solver but with improved scalability. AGMG outperforms AMG and the sparse direct solver in both problems. Figure \ref{fig:Timing_linear} shows the average time spent solving linear systems per time step when $\Delta t = \Delta x$. Reducing $\Delta t$ makes the system matrix more diagonally dominant and as a consequence easier to precondition. 
This effect is noticeable for Problem B using AMG as preconditioner, see Table \ref{Tab:Time}.

\begin{figure}[ht]
	\centering
	\subfloat[
	Average time on solver for Problem A.]
	{\includegraphics[width=0.49\textwidth]{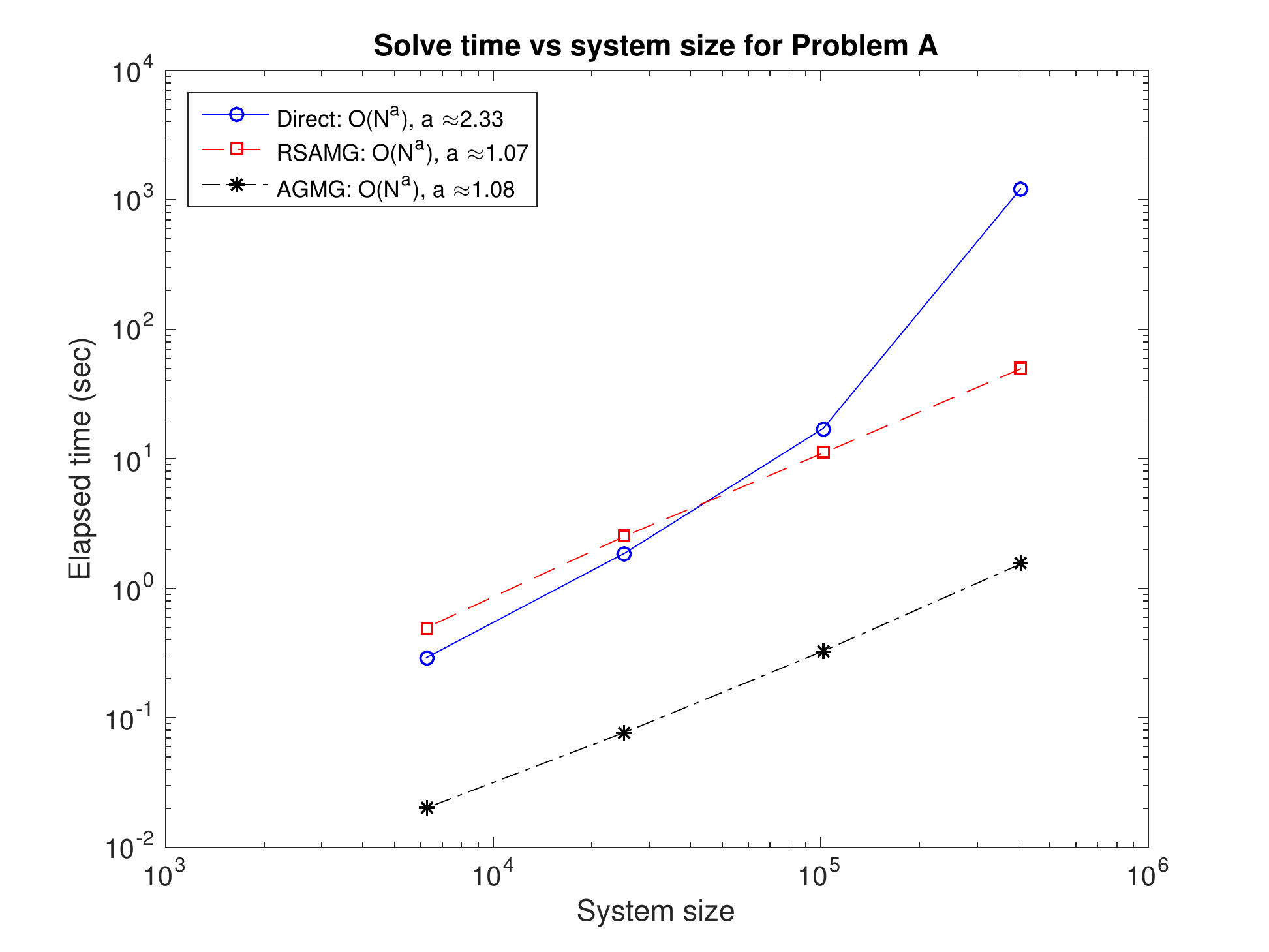}}
		\hspace{\stretch{1}}
		\subfloat[
		Average time on solver for Problem B.]
	{\includegraphics[width=0.49\textwidth]{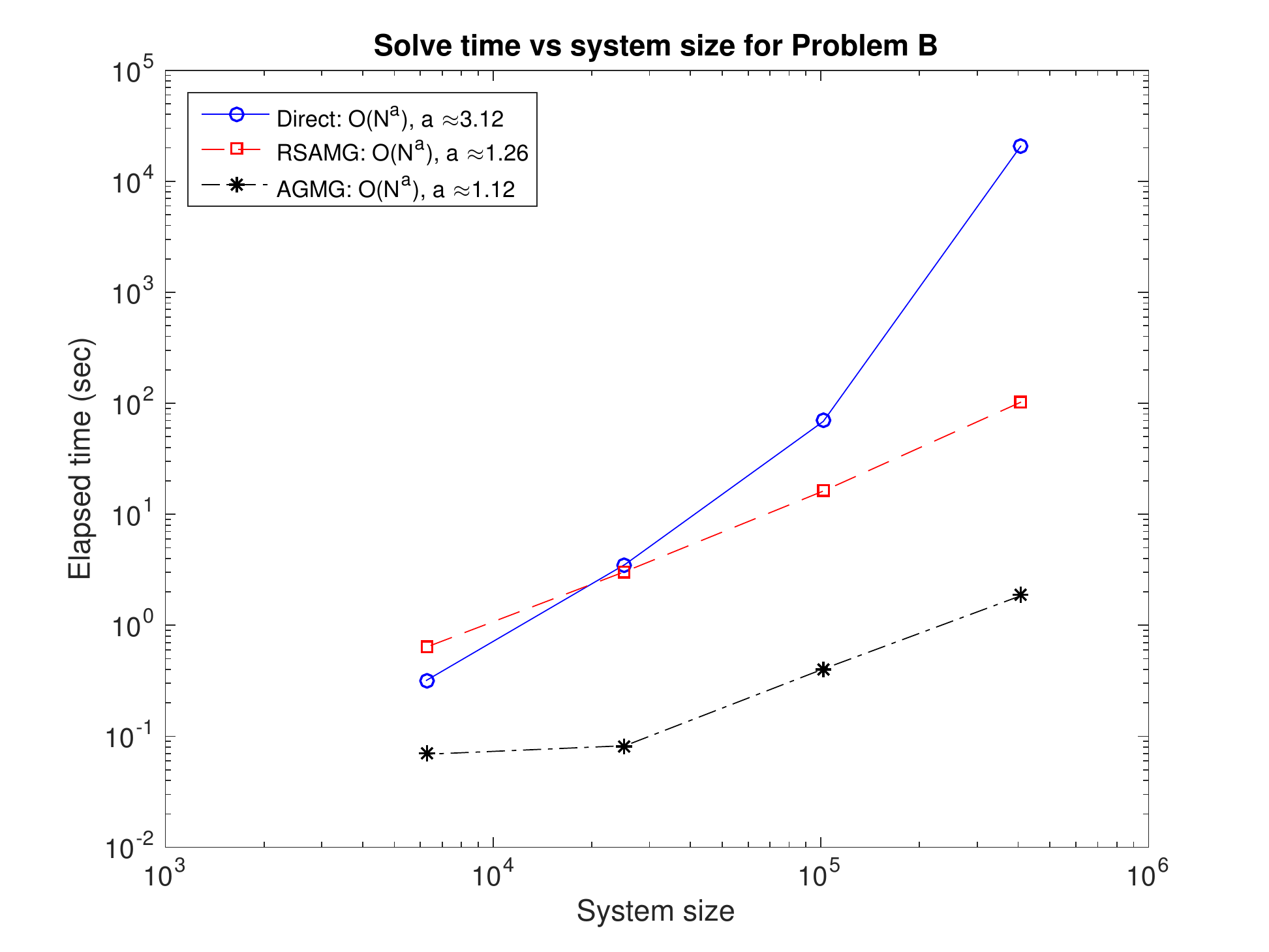}}
	\caption{Average number of seconds per time step for solving the linear systems versus the size of the systems. We use equispaced Cartesian grids in space with $81, 161, 321$ and $641$ nodes per dimension and $\Delta t = \Delta x$. }
	\label{fig:Timing_linear}
\end{figure}

\begin{table}[h]
\captionsetup[subfloat]{position=top}
\begin{center}
\begin{tabular}{c|c|cc|cc|cc}
\multirow{2}{*}{ }& \multirow{2}{*}{$N_x$}&\multicolumn{2}{c|}{Direct}&\multicolumn{2}{c|}{AMG}&\multicolumn{2}{c}{AGMG}\\ 
\cline{3-8}
 & & $\Delta t = T$ & $\Delta t = \Delta x$ & $\Delta t = T$ & $\Delta t = \Delta x$ & $\Delta t = T$ & $\Delta t = \Delta x$ \\
 \hline
\multirow{2}{*}{Problem A} & $321$ & 49.60  & 17.93 & 43.62  & 10.07  & 3.31 & 0.41  \\
& $641$ & 9.50e+03  & 4.33e+03 & 373.77  & 48.82  & 21.22 & 1.84  \\
\hline
\multirow{2}{*}{Problem B} & $321$ & 14.35  & 68.70 & 18.59  & 16.19  & 0.77 & 0.40  \\
& $641$ & 950.62  & 2.09e+04 & 5.64e+03  & 102.57  & 7.82 & 1.85  \\
\end{tabular}
\end{center}
\caption{Average seconds per time step solving linear systems.\label{Tab:Time}}
\end{table}

Table \ref{Tab:Memory} and Table \ref{Tab:Krylov_MG} report memory consumption and quantities related to the Krylov subspace method and to the coarsening. As commented in the previous sections, AMG results in grid and algebraic complexities higher than AGMG's. The coarsening for both of the methods is stopped when the coarse level system is cheap to solve exactly compared to the starting system, specifically, we stop whenever the number of unknowns at the coarse level is comparable to the cubic root of the initial number of unknowns. 
The effect of simplifying the intergrid transfer operators can be observed on the coarse to fine stencil ratio (C/F stencil). For AGMG, the stencil on the coarsest level is similar to the initial one, whereas for AMG it is significantly denser. The fact that aggregation-based coarsening strategies yield coarse matrices with similar sparsity as the original one was noted in \cite{kim2003multigrid}. 
Moreover, AGMG yields shallower hierarchies due to higher coarsening factors.
The effect of reducing $\Delta t$ is also appreciated in this ratio.
We observe that the direct method's consumption increases dramatically while for the MG methods, we note the relation between the memory requirement and the algebraic complexity of the method.

The number of Krylov iterations highlight previous comments on the fact that aggregation-based multigrid methods are not efficient if used as stand-alone solvers: in all test cases, AGMG used more iterations than AMG per policy iteration.
However, AGMG used as a preconditioner to a Krylov subspace method provides accurate solutions faster and cheaper than the other two solvers considered.

\begin{table}[h]
\captionsetup[subfloat]{position=top}
\begin{center}
\subfloat[Peak memory consumption measured in GB for $\Delta t = T$]{
\begin{tabular}{c|c|cc|cc|cc}
\multirow{2}{*}{ }& \multirow{2}{*}{$N_x$}&\multicolumn{2}{c|}{Direct}&\multicolumn{2}{c|}{AMG}&\multicolumn{2}{c}{AGMG}\\ 
\cline{3-8}
 & & VIRT & RES & VIRT & RES & VIRT & RES\\
 \hline
\multirow{2}{*}{Problem A} & $321$ & 11.40  & 5.29 & 10.63  & 5.10  & 9.83 & 4.14  \\
& $641$ & 25.99  & 7.13 & 14.40  & 7.05  & 12.66 & 6.76  \\
\hline
\multirow{2}{*}{Problem B} & $321$ & 11.71  & 6.50 & 9.84  & 5.14  & 9.84 & 5.15  \\
& $641$ & 23.83  & 7.10 & 18.51  & 7.13  & 9.90 & 5.12  \\
\end{tabular}
}
\\
\subfloat[Peak memory consumption measured in GB for $\Delta t = \Delta x$]{
\begin{tabular}{c|c|cc|cc|cc}
\multirow{2}{*}{ }& \multirow{2}{*}{$N_x$}&\multicolumn{2}{c|}{Direct}&\multicolumn{2}{c|}{AMG}&\multicolumn{2}{c}{AGMG}\\ 
\cline{3-8}
 & & VIRT & RES & VIRT & RES & VIRT & RES\\
 \hline
\multirow{2}{*}{Problem A} & $321$ & 8.39  & 3.20 & 6.41  & 2.38 & 6.48  & 2.39  \\
& $641$ & 21.80  & 7.22 & 10.61  & 6.68 & 10.52  & 6.64  \\
\hline
\multirow{2}{*}{Problem B} & $321$ & 13.76  & 7.09 & 9.83  & 3.67 & 9.83  & 3.71 \\
& $641$ & 26.11  & 7.26 & 12.72  & 7.23  & 9.90 & 5.17  \\
\end{tabular}
}
\end{center}
\caption{Peak memory consumption statistics in gigabytes (GB) of the MATLAB process sampled using the shell command {\tt top}. VIRT is the total amount of virtual memory used by MATLAB, whereas RES is the non-swapped physical memory (limited to 7.5).\label{Tab:Memory}}
\end{table}

\begin{table}[h]\captionsetup[subfloat]{position=top}
\begin{center}
\subfloat[Krylov iterations and coarsening related quantities for $\Delta t = T$]{
\begin{tabular}{c|c|c|c|c|c|c|c}
 & Solver & $N_x$ & Avg Krylov It & \# levels & C/F stencil & $c_G$ & $c_A$\\
 \hline
\multirow{4}{*}{Problem A} 
& \multirow{2}{*}{AMG} & $321$ & 4.00  & 9.0 & 26.25 & 2.61 & $7.06^*$  \\
& & $641$ & 4.29  & 11.0 & 22.33  &  2.42 & $4.63^*$  \\
\cline{2-8}
& \multirow{2}{*}{AGMG} & $321$ & 12.67  & 5.83 & 1.30 & 1.81 &  1.97 \\
& & $641$ & 17.14  & 6.14 & 1.36   & 1.61 & 1.77  \\
 \hline \hline
 \multirow{4}{*}{Problem B} 
& \multirow{2}{*}{AMG} & $321$ & 5.00  & 7.0 & 86.26  & 2.08 & $9.60^*$  \\
& & $641$ & 6.67  & 9.5 & 221.43  & 2.23 & $11.61^*$  \\
\cline{2-8}
& \multirow{2}{*}{AGMG} & $321$ & 12.00  & 5.00 & 0.50  & 1.60 &  1.92 \\
& & $641$ & 15.00  & 5.50 & 0.39  & 1.53 & 1.57  \\
\end{tabular}
}
\\
\subfloat[Krylov iterations and coarsening related quantities for $\Delta t = \Delta x$]{
\begin{tabular}{c|c|c|c|c|c|c|c}
 & Solver & $N_x$ & Avg Krylov It & \# levels & C/F stencil & $c_G$ & $c_A$\\
 \hline
\multirow{4}{*}{Problem A} 
& \multirow{2}{*}{AMG} & $321$ & 3.00  & 9.98 & 16.30 & 2.76 & $5.90^*$  \\
& & $641$ & 3.00  & 12.0 & 16.60 &  2.75 & $5.12^*$  \\
\cline{2-8}
& \multirow{2}{*}{AGMG} & $321$ & 6.00  & 2.00 & 0.23 & 1.00 & 1.00 \\
& & $641$ & 6.00 & 2.00 & 0.26 & 1.00 & 1.00  \\
 \hline \hline
 \multirow{4}{*}{Problem B} 
& \multirow{2}{*}{AMG} & $321$ & 2.98  & 8.61 & 48.65  & 2.47 & $8.61^*$  \\
& & $641$ & 2.99  & 11.27 & 107.77  & 2.70 & $11.12^*$  \\
\cline{2-8}
& \multirow{2}{*}{AGMG} & $321$ & 4.99  & 2.96 & 0.31  & 1.07 &  1.02 \\
& & $641$ & 5.01  & 2.97 & 0.33  & 1.15 & 1.10  \\
\end{tabular}
}
\end{center}
\caption{Quantities related the to the Krylov subspace iteration and multigrid coarsening. \textit{Avg Krylov It} contains the average number of Krylov iterations over all time steps and all policy iterations;
\textit{\# levels} contains the average depth in the grid hierarchy; \textit{C/F stencil} contains the ratio between the stencil at the coarsest level and that on the finest level (lower is better). On the finest level, the stencil is close to 11 for Problem A and close to 8 for Problem B. The last two columns report the grid and algebraic complexity as per Definitions \ref{gridc} and \ref{algebrac}.
As the full matrix hierarchy was not available from \cite{HSL} for AMG, but only the coarsest and finest matrices, the starred algebraic complexities are estimates based on the assumption of a geometrically decreasing complexity between the coarsest and finest level, which is likely to be a significant underestimate.
\label{Tab:Krylov_MG}}
\end{table}



\section{Conclusions}
\label{sec:End}
This article discusses two aspects of practical importance associated with wide stencil discretizations of second order non-linear parabolic differential operators. First, we study the truncation of the stencil for problems on bounded domains, as a result of the method overstepping the boundaries for nodes in a surrounding layer. Our main result details the construction of such truncation and proves that the resulting scheme remains consistent, monotone and conditionally stable. Numerical examples confirm that the truncation improves the accuracy of the approach compared to constant and linear extrapolation of the boundary conditions, and the modification of the CFL condition of the scheme.

Second, motivated by the stringent CFL condition of explicit time stepping schemes, we consider implicit schemes and the application of multigrid methods to solve the resulting discrete non-linear system of equations efficiently. Using theoretical and empirical arguments, we show the need to employ multigrid methods based on algebraic ideas. We show that aggregation-based methods are well suited for the discretization schemes considered and justify their use by proving that under mild conditions on the mesh refinement parameters the LISL discretization matrices are M-matrices with non-negative row and column sums. The algorithms are shown to compare favourably against AMG and sparse direct solvers.

Although we only considered linear interpolation, much of the analysis, including in particular the matrix properties in Section \ref{subsec:agmg}, will also hold if other limited interpolations (see, e.g., \cite{debrabant2013semi}, \cite{warin}), are used, as only the properties in (\ref{eq:monotone_int_2}) are critical.

To conclude, we emphasise that monotone schemes for general diffusions in two and more dimensions are necessarily non-local, so that the question of boundary truncation is not restricted to the class of schemes studied in this paper.

\clearpage
\renewcommand{\bibname}{References}
\bibliographystyle{abbrv}  
\bibliography{SL_MG_truncation}

\appendix

\section{Further tables}

\begin{table}[htp]\captionsetup[subfloat]{position=top}
\begin{center}
\newsubtab{
41 & 1.73e-01 & - & 3.91e-02 & - &           3.95e-02 & -  & 3.88e-02 & -\\
81 & 1.39e-01 & 0.32 & 1.84e-02 & 1.09 &     1.84e-02 & 1.10  & 1.83e-02 & 1.09\\
161 & 1.07e-01 & 0.38 & 8.71e-03 & 1.08 &     8.70e-03 & 1.08 & 8.68e-03 & 1.07\\
321 & 8.05e-02 & 0.41 & 1.39e+43 & -150.16 &  4.12e-03 & 1.08 & 4.11e-03 & 1.08\\
641 & 5.95e-02 & 0.44 & 1.77e+153 & -365.76 & 2.17e-03 & 0.92 & 2.17e-03 & 0.92\\
}
{Error in $L^\infty$-norm over $\Omega_{\Delta x}$}{}
\\\newsubtab{
41 & 5.71e-02 & - & 3.91e-02 & - & 			 3.95e-02 & - & 3.88e-02 & -\\
81 & 2.74e-02 & 1.06 & 1.84e-02 & 1.09 & 	 1.84e-02 & 1.10 & 1.83e-02 & 1.09\\
161 & 1.31e-02 & 1.06 & 8.71e-03 & 1.08 & 	  8.70e-03 & 1.08& 8.68e-03 & 1.07\\
321 & 6.57e-03 & 0.99 & 8.34e+28 & -102.92 &  4.12e-03 & 1.08& 4.11e-03 & 1.08\\
641 & 3.28e-03 & 1.00 & 1.09e+127 & -325.93 & 2.17e-03 & 0.92 & 2.17e-03 & 0.92\\
}
{Error in $L^\infty$-norm over $ \Omega_{\Delta x} \cap \lbrack -\pi/2 , \pi/2 \rbrack^2$}{}
\end{center}
\caption{Results using stencil truncation for explicit method with $N_{\alpha} = 40$ for Problem B.\label{Tab:Linf_exp_trunc_40_B}}
\end{table}

\begin{table}[htp]\captionsetup[subfloat]{position=top}
\begin{center}
\newsubtab{
41 & 1.25e+00 & - & 3.79e-01 & - &         3.82e-01  & - & 3.75e-01 & -\\
81 & 1.99e+00 & -0.67 & 3.55e-01 & 0.09 &  3.55e-01  & 0.11 & 3.53e-01 & 0.09\\
161 & 3.04e+00 & -0.61 & 2.92e-01 & 0.28 &  2.92e-01 & 0.28 & 2.92e-01 & 0.27\\
321 & 4.52e+00 & -0.57 & 2.35e-01 & 0.32 &  2.35e-01 & 0.32 & 2.35e-01 & 0.31\\
641 & 6.62e+00 & -0.55 & 1.77e-01 & 0.41 &  1.77e-01 & 0.41 & 1.77e-01 & 0.41\\
}
{Error in $L^\infty$-norm over $\Omega_{\Delta x}$}{Tab:Ex1a}
\\\newsubtab{
41 & 5.71e-02 & - & 6.38e-02 & - &         6.34e-02 & -   & 6.40e-02 & -\\
81 & 2.74e-02 & 1.06 & 5.72e-02 & 0.16 &   5.72e-02 & 0.15   & 5.68e-02 & 0.17\\
161 & 1.31e-02 & 1.06 & 4.51e-02 & 0.34 &   4.51e-02 & 0.34  & 4.49e-02 & 0.34\\
321 & 6.57e-03 & 0.99 & 3.71e-02 & 0.28 &   3.71e-02 & 0.28  & 3.70e-02 & 0.28\\
641 & 3.28e-03 & 1.00 & 2.89e-02 & 0.36 &   2.89e-02 & 0.36  & 2.88e-02 & 0.36\\
}
{Error in $L^\infty$-norm over $ \Omega_{\Delta x} \cap \lbrack -\pi/2 , \pi/2 \rbrack^2$}{Tab:Ex1b}
\end{center}
\caption{Results using constant extrapolation of the boundary condition for explicit method with $N_{\alpha} = 40$ for Problem B.\label{Tab:Linf_exp_const_40_B}}
\end{table}

\begin{table}[htp]\captionsetup[subfloat]{position=top}
\begin{center}
\newsubtab{
41 & 5.71e-02 & - & 8.46e-02 & - &           8.29e-02 & -     & 8.60e-02 & - \\
81 & 3.12e-02 & 0.87 & 2.43e+02 & -11.49 &   1.82e+02 & -11.10    & 1.67e+03 & -14.25\\
161 & 2.89e-02 & 0.11 & 7.90e+18 & -54.85 &   8.95e+20 & -62.10   & 1.64e+31 & -92.99\\
321 & 2.38e-02 & 0.28 & 1.51e+70 & -170.36 &  9.26e+93 & -242.55  & 1.51e+164 & -441.69\\
641 & 1.87e-02 & 0.35 & 1.14e+207 & -454.70 & NaN & NaN           & NaN & NaN\\
}
{Error in $L^\infty$-norm over $\Omega_{\Delta x}$}{Tab:Ex1a}
\\\newsubtab{
41 & 5.71e-02 & - & 3.91e-02 & - &           3.95e-02 & -     & 3.88e-02 & - \\
81 & 2.74e-02 & 1.06 & 1.84e-02 & 1.09 &     1.84e-02 & 1.10     & 5.19e-02 & -0.42\\
161 & 1.31e-02 & 1.06 & 1.36e+09 & -36.10 &   1.54e+11 & -42.92  & 2.82e+21 & -75.52\\
321 & 6.57e-03 & 0.99 & 5.30e+52 & -144.81 &  3.24e+76 & -217.00 & 5.27e+146 & -416.15\\
641 & 3.28e-03 & 1.00 & 6.39e+176 & -412.19 & NaN & NaN          & NaN & NaN\\
}
{Error in $L^\infty$-norm over $ \Omega_{\Delta x} \cap \lbrack -\pi/2 , \pi/2 \rbrack^2$}{Tab:Ex1b}
\end{center}
\caption{Results using linear extrapolation for points out of the domain for explicit method with $N_{\alpha} = 40$ for Problem B.\label{Tab:Linf_exp_lin_40_B}}
\end{table}

\begin{table}[htp]\captionsetup[subfloat]{position=top}
\begin{center}
\newsubtab{
41 & 3.00e-02 & - & 3.76e-02 & - &        3.72e-02 & -  & 3.79e-02 & -\\
81 & 1.40e-02 & 1.10 & 1.80e-02 & 1.06 &  1.80e-02 & 1.05  & 1.45e-02 & 1.38\\
161 & 6.34e-03 & 1.15 & 6.36e-03 & 1.50 &  6.37e-03 & 1.50 & 7.72e-03 & 0.91\\
321 & 3.04e-03 & 1.06 & 3.38e-03 & 0.91 &  3.50e-03 & 0.86 & 3.01e-03 & 1.36\\
641 & 1.53e-03 & 0.99 & 1.77e-03 & 0.93 &  1.76e-03 & 1.00 & 1.66e-03 & 0.85\\
}
{Error in $L^\infty$-norm over $\Omega_{\Delta x}$}{Tab:Ex1a}
\\\newsubtab{
41 & 3.00e-02 & - & 3.76e-02 & - &       3.72e-02 & -  & 3.79e-02 & -\\
81 & 1.40e-02 & 1.10 & 1.80e-02 & 1.06 & 1.80e-02 & 1.05  & 1.45e-02 & 1.38\\
161 & 6.34e-03 & 1.15 & 6.31e-03 & 1.51 & 6.37e-03 & 1.50 & 7.72e-03 & 0.91\\
321 & 3.04e-03 & 1.06 & 3.38e-03 & 0.90 & 3.50e-03 & 0.86 & 3.01e-03 & 1.36\\
641 & 1.53e-03 & 0.99 & 1.77e-03 & 0.93 & 1.76e-03 & 1.00 & 1.66e-03 & 0.85\\
}
{Error in $L^\infty$-norm over $\Omega_{\Delta x} \cap \lbrack -\pi/2 , \pi/2 \rbrack^2$}{Tab:Ex1b}
\end{center}
\caption{Results using truncation for points out of the domain for implicit method with $N_{\alpha} = 40$ for Problem B.\label{Tab:Linf_imp_trunc_40_B}}
\end{table}

%
%


\begin{table}
{\small
\begin{tabular}{c|c|cc|cc|cc}
\multirow{2}{*}{}  &\multirow{2}{*}{$N_x$} & \multicolumn{2}{c|}{Direct} & \multicolumn{2}{c|}{AMG} & \multicolumn{2}{c}{AGMG} \\ \cline{3-8}
&  & $\Delta t = T$ &	$\Delta t = \Delta x$ &	$\Delta t = T$ &	$\Delta t = \Delta x$ &	$\Delta t = T$ &	$\Delta t = \Delta x$ \\ \hline
\multirow{3}{*}{Problem A}&	161 &	44.17\% &	69.93\% &	64.21\% &	75.42\% &	12.40\% &	7.63\% \\
& 321	& 53.34\% &	82.36\% &	69.24\% &	77.24\% &	15.09\% &	10.82\% \\
& 641	& 78.01\% &	85.91\% &	34.20\% &	67.47\% &	5.53\% &	10.37\% \\ \hline
\multirow{3}{*}{Problem B}	 & 161 &	8.46\% &	80.71\% &	49.16\% &	79.35\% &	7.27\% &	8.92\% \\
& 321 &	26.09\% &	94.36\% &	77.48\% &	76.28\% &	12.88\% &	8.83\% \\
& 641 &	95.25\% &	97.65\% &	98.64\% &	87.68\% &	2.48\% &	17.06\%
\end{tabular}
}
\caption{Percentage of computational time spent in linear solvers for the Examples in Section \ref{subsec:perform}.}
\label{Tab:solverpercentage}
\end{table}

\end{document}